\documentclass[12pt, english]{article}
\usepackage[T1]{fontenc}
\usepackage[utf8]{inputenc}
\usepackage{amsmath}
\usepackage{amsthm}
\usepackage{amssymb}
\usepackage{enumerate}
\usepackage{color}
\usepackage[%
  pagebackref, %backref with page and hyper
  breaklinks, %allow line break in toc
  colorlinks=true, linkcolor=blue, citecolor=blue, urlcolor=blue
]{hyperref}

\numberwithin{equation}{section}
\makeatletter
\usepackage{tikz-cd}

\newcommand{\pzbox}{%
	\mathbin{%
		\tikz[baseline=-0.05em, scale=0.24, line width=0.06em]{%
			\draw (0,0) rectangle (1,1);
			\draw (0,1) -- (1,0);
		}%
	}%
}

\theoremstyle{plain}
\theoremstyle{remark}

\usepackage[blocks]{authblk}
\usepackage{fullpage}
\usepackage{makeidx}
\usepackage{amsfonts}
\usepackage{latexsym}
\usepackage{babel}
\makeatother

\allowdisplaybreaks

\def \1{{\bf 1}}
\def\C{{\mathbb{C}}}
\def\Z{{\mathbb{Z}}}

\def\N{{\mathbb{N}}}

\def\Res{\mathrm{Res}}

\def\wt{\mathrm{wt}}

\def\G{{\mathcal{G}}}
\def\H{{\mathcal{H}}}
\def\Hom{\operatorname{Hom}}

\def\Y{{\mathcal{Y}}}

\def\<{\langle}
\def\>{\rangle}

\theoremstyle{definition}
\newtheorem{lemma}{Lemma}[section]
\newtheorem{theorem}[lemma]{Theorem}
\newtheorem{corollary}[lemma]{Corollary}
\newtheorem{proposition}[lemma]{Proposition}
\newtheorem{definition}[lemma]{Definition}

\newtheorem{example}[lemma]{Example}
\newtheorem{remark}[lemma]{Remark}
\usepackage{times}
\newtheorem{mainthm}{Theorem}

\title{Finiteness and Construction of Internal Hom for Vertex Operator Algebras}

\author{Chao Yang \\ Email: yangcabc@163.com }
\affil{ School of Mathematics,  Southwest Jiaotong University, Chengdu 611756, China}

\author{Yiyi Zhu  \\  Email: yzhu51@gdut.edu.cn  }     
\affil{School of Mathematics and Statistics, Guangdong University of Technology,  Guangzhou 510520, China}

\date{}

\begin{document}
\maketitle
	
\begin{abstract}

Let $V$ be a vertex operator algebra, and let $W^1$ and $W^2$ be restricted $V$-modules. 
We construct a generalized $V$-module $ \H(W^1, W^2)$ characterized by canonical universal properties. 
We show that, under suitable hypotheses, $ \H(W^1, W^2)$ realizes the internal Hom object in the tensor category of restricted $V$-modules. Although our construction differs from Li's, we show that it agrees with the natural logarithmic generalization of 
Li's module $\Delta(W^1, W^2)$. 
We further establish a canonical isomorphism between $\mathcal H\big(W^1,(W^2 )^\prime \big)$ and the $P(z_0)$-dual product 
$ W^1 \pzbox_{P(z_0)} W^2 $  recently  constructed by Du and Huang.
Under the $C_1$-cofiniteness condition, 
we investigate finiteness properties of $ \mathcal H(W^1, W^2)$. 
As applications, we obtain a natural isomorphism between $ \mathcal H(W^1, W^2)'$ and $  W^1 \boxtimes (W^2)'$,
and prove the finiteness of the corresponding fusion rules. 

\end{abstract}

\tableofcontents

\section{Introduction}

Let $(\mathcal{C}, \otimes)$ be a monoidal category. For objects $W_1, W_2 \in \mathcal{C}$,
an object $\underline{\text{Hom}}(W_1, W_2) \in  \mathcal{C}$  is called the internal Hom from $W^1$ to $W^2$ if there exists a natural isomorphism
\[ 
\Hom_{\mathcal{C}} \big(X, \; \underline{\Hom}(W_1, W_2) \big) 
\cong \Hom_{\mathcal{C}}( X \otimes W_1, \; W_2),
\]
for all $X \in \mathcal{C}$. 
Now let $V$ be a vertex operator algebra, 
and let $\mathcal{C}_V$ denote the category of restricted $V$-modules of finite composition length.
It is known that  $\mathcal{C}_V$ admits a tensor category structure under appropriate assumptions \cite{CKM24, H08, McR24, CMHFY26}.
Given $W_1, W_2 \in \mathcal{C}_V$, a natural problem is to determine the internal Hom and find an explicit construction.

Let $V$ be a regular vertex operator algebra. The module $\Delta(W^1, W^2)$ constructed in \cite{Li98} 
realizes the internal Hom in $\mathcal{C}_V$, although this terminology was not used there.
We briefly recall the construction.  
Let $W^1, W^2 \in \mathcal{C}_V$. 
Li first established in \cite{Li98} that the space  $\G(W^1, W^2)$ 
(see our definition \ref{G-inter}) of generalized intertwining operators between $W^1$ and $W^2$
carries a natural weak $V$-module structure. 
He then define $\Delta(W^1, W^2)$ as the sum of all ordinary $V$-submodules of $\G(W^1, W^2)$.
Under the assumptions that $V$ is rational and the fusion rules among irreducible modules are finite,
Li further showed that $\Delta(W^1, W^2)$ is also an ordinary $V$-module (hence $\H(W^1, W^2) \in \mathcal{C}_V$)
and satisfies a canonical universal property. 
Together with the universal property of the tensor product, 
this ensures that $\Delta(W^1,W^2)$ is exactly the internal Hom object in $\mathcal{C}_V$.

One main objective of this paper is to generalize the relevant results in \cite{Li98} to the general non-rational framework. For this purpose,  the most straightforward approach would be extending Li’s construction of $\Delta(W^1, W^2)$ to the logarithmic setting. In this paper, however, we adopt a different construction and denote the resulting object by  $\mathcal{H}(W^1, W^2)$. One main advantage is that we do not need to first verify the $V$-module structure of $\G(W^1, W^2)$,
since this argument is rather lengthy \cite{Li96, Li98, DLXY24}.
Moreover, this construction helps establish an isomorphism between  $\mathcal{H}\big(W^1, (W^2)^\prime \big)$ 
and the space $W^1 \pzbox_{P(z_0)} W^2$, recently constructed by Du and Huang.

Let $V$ be a vertex operator algebra, and let $W^1, W^2$ be restricted $V$-modules.
We define 
\begin{align*}
	\mathcal{H}(W_1, W_2) & = 
    \left\{ \Y_A(a, t) \ \big| \ A  \ \text{is a restricted $V$-module}, \; \; a \in A, \  \Y_A \in \mathcal{I}\binom{W_2}{A\ W_1} \right\}\\
	          & \subset \text{Hom}_{\C} \big(W_1,  W_2[\log t] \{t\} \big).
\end{align*}
For $u \in V, n \in \Z$, and $\Y_{A}(a, t) \in  \mathcal{H}(W^1,W^2)$, 
we define
\[
u^{\mathcal{H}}_n  \Y_{A}(a, t) =  \Y_A(v_na, t) \in \mathcal{H}(W^1,W^2).
\]
Since each $A$ is a restricted $V$-module, $\H(W^1, W^2) $ carries a natural $V$-module structure under this action.
We further define a linear map
\[
\Y^{\H}(-,x): \H(W^1,W^2) \otimes W^1 \to W^2 [\log x] \{x\}  
\]
by
\[
\Y^{\H} \big(\Y_A(a, t), x \big)w^1=\Y_A(a, x)w^1.
\]

Our first main result is:

\begin{mainthm} 
Under the above definition,  $\mathcal{H}(W_1,W_2)$ forms a generalized $V$-module, and 
$\mathcal{Y}^{\mathcal{H}}$ is an intertwining operator of type 
$ \binom{W_2}{\mathcal{H}(W_1, W_2)\ W_1}$. 
Moreover, $\big( \mathcal{H}(W_1,W_2),  \mathcal{Y}^{\mathcal{H}} \big)$ is characterized by a certain universal property.
In particular, $  \mathcal{H}(W_1,W_2)$ coincides with $ \Delta(W_1,W_2)$ in the non-logarithmic case.

\end{mainthm}

Note that the definition only guarantees that $\mathcal{H}(W^1, W^2)$ is a generalized $V$-module, while it is not clear whether it is grading-restricted. Establishing the restrictiveness of $\mathcal{H}(W^1, W^2)$ under suitable assumptions is therefore essential. First, in order to relate $\mathcal{H}(W^1, W^2)$ to the tensor product $W^1 \boxtimes (W^2)'$, we make use of the operator $A_0\Omega_0\mathcal{Y}^{\mathcal H}$, which is well defined only when $\mathcal{H}(W^1, W^2)$ is a restricted module. Second, to show that $\mathcal{H}(W^1, W^2)$ realizes the internal Hom object in the category $\mathcal C_V$, one must first verify that $\mathcal{H}(W^1, W^2)$ itself belongs to $\mathcal C_V$.

Since tensor products exist and remain $C_1$-cofinite under the $C_1$-cofinite condition \cite{M14, H25C1, YZ26}, together with the relation between $\Y^\H$ and $\mathcal{Y}^\boxtimes$, we derive the second main result of this paper:

\begin{mainthm}\label{mainthmB}
Let $W^1$ and $ (W^2)^\prime $ be $C_1$-cofinite restricted $V$-modules. Then
\begin{enumerate}[{(1)}]
 
\item  $\mathcal{H}(W^1, W^2)'$ is also a $C_1$-cofinite restricted $V$-module and isomorphic to $W^1 \boxtimes (W^2)'$.
In particular $\H(W^1, W^2)$ itself is a restricted $V$-module. 

\item For any restricted $V$-module $W$, the spaces of intertwining operators  $ \mathcal{I}\binom{W^2}{W \; W^1} $ and   $ \mathcal{I}\binom{W}{W^1 \; (W^2)^\prime} $ are both finite-dimensional. 

\end{enumerate}

\end{mainthm}

Let $V$ be a vertex operator algebra for which a restricted $V$-module is $C_1$-cofinite if and only if it has finite composition length.
With this condition, $\mathcal{C}_V$ carries a natural tensor category structure\cite{McR24, CMHFY26}.
Having established the finiteness of $\mathcal{H}(W^1, W^2)$, the universal properties of $\mathcal{H}(-,-)$ and the tensor product immediately yield our third main result:

\begin{mainthm}
Let $W^1, W^2 \in \mathcal{C}_V$. Then $\H(W^1, W^2) $ also lies in $\mathcal{C}_V$, and there is a natural isomorphism
\[
\operatorname{Hom}_{V} \big(X, \; \H(W^1, W^2) \big) \cong \operatorname{Hom}_{V}( X \boxtimes W^1, \; W^2) 
\]
for all $X \in \mathcal{C}_V$.  Consequently, $\mathcal{H}(W^1,W^2)$ is the internal Hom from $W^1$ to $W^2$ in $\mathcal{C}_V$.
\end{mainthm}

Fix a nonzero complex number $z_0$.
We finally clarify the relationship between $\H\big(W^1, (W^2)^\prime \big)$ and the $P(z_0)$-dual product $W^1 \pzbox_{P(z_0)}W^2$.
In \cite{DH25, H26},  the $P(z_0)$-dual product $W_1 \pzbox_{P(z_0)} W_2$  is defined for twisted modules 
as a subspace of $(W_1 \otimes W_2)^*$. 
One of the main motivations for introducing this construction is to realize the tensor product construction.
In this paper, we adopt its special case for untwisted modules.
It was shown in \cite{H26} that if $W_1$ and $W_2$ are $C_1$-cofinite restricted $V$-modules, 
then $W^1 \pzbox_{P(z_0)} W^2$ is also a restricted $V$-module,
and the tensor product $W^1 \boxtimes W^2 $ is isomorphic to $\big( W^1 \pzbox_{P(z_0)} W^2 \big)^\prime$.
Combining this fact with Theorem \ref{mainthmB}, we obtain that for any $C_1$-cofinite restricted $V$-modules $W_1$ and $W_2$, the modules $W^1 \boxtimes_{P(z_0)} W^2$ and $\mathcal{H}\big(W^1, (W^2)^\prime\big)$ are isomorphic. In fact, this isomorphism remains valid without the $C_1$-cofiniteness assumption, as shown in the following theorem:

\begin{mainthm}
Let $W^1,W^2$ be restricted $V$-modules.
Then $\H \big(W^1, (W^2)^\prime \big)$ and $W^1\pzbox_{P(z_0)} W^2$ are isomorphic as generalized $V$-modules.
\end{mainthm}

This paper is organized as follows: 
Section \ref{sec2} recalls fundamental concepts related to restricted $V$-modules, as well as their contragredient modules,
and the properties of $C_1$-cofinte modules.
Section \ref{sec3} recalls the definition of intertwining operators and collects several auxiliary lemmas needed later.  
In Section \ref{sec4}, we introduce our main construction $\mathcal{H}(W^1, W^2)$ and establish its fundamental properties.  
Several explicit examples are supplied, and we show that this construction coincides with the natural logarithmic generalization of Li’s module 
$\Delta(W^1, W^2)$.
Section \ref{sec5} is devoted to investigating finiteness properties under the $C_1$-cofiniteness condition.  
We prove that $\mathcal{H}(W^1, W^2)'$ is isomorphic to $ W^1 \boxtimes (W^2)'$ and that the corresponding fusion rules are finite. 
Moreover, under suitable hypotheses, we show that $\mathcal{H}(W^1, W^2)$ is exactly the internal Hom object in the tensor category of restricted 
$V$-modules.  
In Section \ref{sec6}, we establish a canonical isomorphism between $\mathcal{H}(W^1, (W^2)')$ and the $P(z_0)$-dual product $ W^1 \pzbox_{P(z_0)} W^2$ constructed by Du and Huang.
    
\section{Preliminaries}\label{sec2}

    Throughout this paper, the symbols $x,x_0,x_1,x_2,y,t$ stand for pairwise commuting formal variables, whereas $z$ is a complex variable and $z_0$ stands for a nonzero complex number.
    
	We adopt the standard framework of vertex operator algebras as presented in \cite{FLM88, LL04}, 
	using the notation $(V, Y(-, x), \mathbf{1}, \omega)$ for a vertex operator algebra. 
	By a {\bf weak $V$-module}, we mean a vector space $W$ equipped with a linear map
	$Y_W(-,x)-: V \otimes W \to W((x))$
	satisfying $Y_W(\mathbf{1}, x) = \operatorname{id}_W$ and the Jacobi identity.
	  	
	Let $W$ be a weak $V$-module.
	For any $\lambda \in \mathbb{C}$, let $W_{(\lambda)}$ denote the generalized $\lambda$-eigenspace of the action of $L(0)$ on $W$.
	When $W_{(\lambda)}$ is nonzero, $W_{(\lambda)}$ is also called a {\bf weight space} of $W$, and $\lambda$ is called a {\bf weight} of $W$.
\begin{definition}

\begin{enumerate}[{(1)}]

\item A {\bf generalized $V$-module} is a weak $V$-module $W$ equipped with the weight space decomposition
\[
W = \bigoplus_{\lambda \in \mathbb{C}} W_{(\lambda)}.
\]

\item A {\bf lower-truncated generalized $V$-module} is a generalized $V$-module $W$ such that 
for every $\lambda \in \mathbb{C}$, $W_{(\lambda + n)}=0$ for all sufficiently negative integers $n$. 

\item A {\bf grading restricted generalized  $V$-module} (or simply a {\bf restricted $V$-module}) is a 
lower-truncated generalized $V$-module whose weight spaces are all finite-dimensional.

\item  An {\bf ordinary $V$-module} is a restricted $V$-module on which $L(0)$ acts semisimply. 
\end{enumerate}

\end{definition}

This paper also involves standard concepts for vertex operator algebras, such as 
$C_2$-cofiniteness and rationality; we refer the reader to \cite{Z96,DLM97} for these notions.
    
\begin{remark}
The definition of ordinary $V$-module here is the same as in \cite{DLM97, ABD04}.
Although some recent papers (e.g., \cite{DRX17,DNR25,DRX24}) refer to such modules simply as $V$-modules.
\end{remark}

Let $W$ be a  generalized $V$-module. 
A subspace $U \subseteq W$ is said to be {\bf homogeneous} if 
\[
U = \bigoplus_{\lambda \in \mathbb{C}} \big( U \cap W_{(\lambda)} \big).
\]
For a homogeneous subspace $U$, we define the set
\[
\wt (U) =\left\{ \lambda \in \C \   \big| \  U \cap W_{(\lambda)} \neq 0   \right\}.
\]
We remark that a homogeneous subspace need not be invariant under $L(0)$.
	
\begin{definition}
	A weak $V$-module $W$ is said to be $C_1$-cofinite if the quotient space $W/C_1(W)$ is finite-dimensional,
	where
	\[
	C_1(W)=\operatorname{span}\left\{ v_{-1}w \ | \ v \in V,    \  \wt (v) >0,  \ w\in W  \right\}.
	\]	      
\end{definition}
Note that for any generalized $V$-module $W$, the subspace $C_1(W)$ is homogeneous. Hence $C_1(W)$ has a homogeneous complement in $W$. 
The following result is standard.
	
\begin{lemma} \label{C1-lemma}
    Let $W$ be a $C_1$-cofinite, lower-truncated generalized $V$-module. 
    Let $E_W$ be a homogeneous complement of $C_1(W)$ in $W$.
    Then 
    \[
    W= \operatorname{span} 
    \left\{m, \; v^{1}_{-1} \cdots v^{s}_{-1} w \; \big| \; v^i \in V, \; \wt (v^i) >0, \; i=1, \cdots ,s, \; s \geq 1, \; m, w\in E_W  \right\}.
    \]
    Let $ \wt (E_W) =\{\lambda_1, \cdots, \lambda_n \}.$
    Then 
    \[
     \operatorname{wt}(W) \subset \bigcup_{i=1}^n (\lambda_i + \mathbb{N}).
    \]
    In particular, $W$ is a finitely generated restricted $V$-module.
\end{lemma}

\begin{lemma} \label{C1-exact}
Let 
\[
0 \to W^1 \xrightarrow{f}  W^2   \xrightarrow{g}   W^3  \to 0
\]
be an exact sequence of lower-truncated generalized $V$-modules.
Then the following statements hold.
\begin{enumerate}[{(1)}]
    \item If $W^2$ is $C_1$-cofinite, then $W^3$ is also $C_1$-cofinite;
    \item If $W^1, W^3$ are both $C_1$-cofinite, then $W^2$ is also $C_1$-cofinite.
\end{enumerate}
\end{lemma}

\begin{proof}
See \cite[Lemma 2.10 and 2.11]{H09}.
\end{proof}

The following well-known results, which do not usually appear together in a single reference, 
are collected here (see, e.g., \cite{ABD04, M04, H09, McR26}).
It shows that the results of this paper, which are established under the $C_1$-cofiniteness condition, apply naturally to restricted modules over $C_2$-cofinite vertex operator algebras.

\begin{lemma}\label{lem:C2-C1}
Let $V$ be a $C_2$-cofinite vertex operator algebra without negative conformal weights.
Then every weak $V$-module is automatically a lower-truncated generalized $V$-module (see, e.g., \cite{ABD04, M04}).
Let $W$ be a lower-truncated generalized  $V$-module. Then the following statements are equivalent:
\begin{enumerate}[{(1)}]
    \item $W$ is restricted;
    \item $W$ is of finite composition length;
    \item $W$ is $C_1$-cofinite;
    \item $W$ is finitely generated.
\end{enumerate}

\end{lemma}

\begin{proof}
$ (1) \Rightarrow (2)$ \
It is well known that $V$ admits only finitely many mutually non-isomorphic irreducible ordinary $V$-modules; see, e.g., \cite{DLM00}.  
Let $S_1, S_2,\dots, S_m$ be a complete set of representatives of these isomorphism classes, and $\lambda_1,\lambda_2,\dots,\lambda_m$ be their conformal weights, respectively.

For a restricted $V$-module $W$, define a non-negative integer $d_W$ by
\[
d_W=\dim W_{(\lambda_1)} + \cdots +\dim W_{(\lambda_m)}.
\]
Since every lower-truncated generalized  $V$-module has an irreducible ordinary subquotient, it follows that such a module $W$ is zero if and only if $d_W=0$.
Now, by induction on $d_W$, one can readily prove that any restricted $V$-module $W$ has finite composition length bounded above by $d_W$.

\vspace{0.5\baselineskip} 

$ (2) \Rightarrow (3)$ \
By Lemma \ref{C1-exact} (2), we only need to verify the claim for irreducible ordinary $V$-modules.
Any irreducible ordinary $V$-module $W$ can be generated by a single element $w$.
Write $V=E \oplus C_2(V)$, where $E$ is a finite-dimensional homogeneous subspace of $V$.
By the spanning set theorem in \cite{Bu02, M04}, we have 
\[
W= \operatorname{span}_{C} \left\{v^1_{n_1} v^2_{n_2} \cdots v^p_{n_p} w \; \big{|} \;  v^1, \cdots, v^p \in E,   \; n_1 < n_2 < \cdots < n_p  \right\}.
\]
Combining the lower-truncated property and the finite dimensionality of $E$, 
there exists a positive integer $N$ such that $v_n w=0$ for all $v\in E$ and all integers $n \geq N$.
Let $E_W$ be a homogeneous complement of $C_1(W)$ in $W$. Then 
\[
E_W = \operatorname{span}_{C} \left\{v^1_{n_1} v^2_{n_2} \cdots v^p_{n_p} w \; \big{|} \; v^1, \cdots, v^p \in E,   \;  -1 < n_1 < n_2 < \cdots < n_p <N \right\}.
\]
It is obvious that $E_W$ is finite-dimensional,  which implies that $W$ is $C_1$-cofinite.

\vspace{0.5\baselineskip} 

$ (3) \Rightarrow (4)$ \ It follows from Lemma \ref{C1-lemma}.

\vspace{0.5\baselineskip} 

$ (4) \Rightarrow (1)$ \  
%The argument below follows the idea in \cite{M04, McR26}.
It is enough to consider the case where $W$ is generated by a single homogeneous element $w$, and we suppose  $w \in W_{(\lambda)}$.
Again  by the spanning set theorem in \cite{Bu02, M04}, we have 
\[
W= \operatorname{span}_{\C} \left\{v^1_{n_1} v^2_{n_2} \cdots v^s_{n_s} w \; \big{|} \;  \text{homogeneous} \; v^1, \cdots, v^s \in E,  \; n_1 < n_2 < \cdots < n_s <N \right\},
\]
where $N$ is defined as above. 
Consequently, for any integer $d$, the homogeneous component satisfies
\[
\begin{aligned}
W_{(\lambda+d)}
=\operatorname{span}_{\mathbb{C}}
\biggl\{
v^1_{n_1}v^2_{n_2}\cdots v^s_{n_s}w
\;\bigg|\;
&v^1,\dots,v^s\in E,\; n_1<n_2<\dots<n_s <N, \\
&\sum_{i=1}^s \big(\operatorname{wt}v^i - n_i - 1\big) = d
\biggr\}.
\end{aligned}
\]
Since $E$ is finite-dimensional and $N$ is fixed, for any given integer $d$, 
there exist only finitely many tuples  $ (s; n_1,\dots,n_s)$ with $s \geq 0  $ and $ n_1<\cdots<n_s<N $  satisfying the weight condition
\[
\sum_{i=1}^s \bigl(\operatorname{wt} v^i - n_i - 1\bigr) = d.
\]
Thus, each homogeneous subspace $W_{(\lambda+d)} $ is finite-dimensional.
Therefore, $W$ is restricted $V$-module.

The proof is complete.

%the weight condition
%\[
%\sum_{i=1}^s \bigl(\operatorname{wt} v^i - n_i - 1\bigr) = d
%\]
%has only finitely many tuples $ (s; n_1,\dots,n_s)$ satisfying $s \geq 0  $ and $ n_1<\cdots<n_s<N $. 
%Thus, each homogeneous subspace $W_{(\lambda+d)} $ is finite-dimensional.
%Therefore, $W$ is restricted $V$-module.

%The proof is complete.
\end{proof}

{\bf Definition of $x^{\pm L(0)}$  :}  \; 	Let $W$ be a  generalized $V$-module.  
	The semisimple part $L(0)_s$ of $L(0)$ is defined as the linear operator on $W$  such that
	\[
	L(0)_s (w) =\lambda w,  \; \; \; \forall \;  w \in W_{(\lambda)},  \;   \lambda \in \C.
	\]
    Accordingly, the nilpotent part of $L(0)$ is defined by 
    \[
    L(0)_n = L(0) - L(0)_s.
    \]
	The operators $L(0)_s$ and $L(0)_n$ commute.
    Since $L(0)$ acts locally finite on $W$, the formal operator powers $x^{\pm L(0)}$ are well-defined (see \cite{HLZ10II}). 
	We recall their definitions: $x^{\pm L(0)}$ are the linear maps from $W$ to $W[\log x] \{x\}$ given by
   \begin{align*}
    x^{\pm L(0)}w  & = x^{\pm L(0)_n} x^{\pm L(0)_s}w = x^{\pm \lambda}  e^{\pm L(0)_n \log x} w\\
    & =   \sum_{k \geq 0} \frac{1}{k!} (\pm 1)^k \big(L(0)_n \big)^k  w  x^{\pm \lambda} (\log x )^k \in x^{\lambda}  W[\log x]
    \end{align*}
	for any homogeneous $w \in W_{(\lambda)}$. It is straightforward to verify that
    \[
    x^{L(0)} x^{-L(0)}=id_W= x^{-L(0)} x^{L(0)}.
    \]
	
	\vspace{0.5\baselineskip}

   {\bf Contragredient modules:} \ We recall the definition of the contragredient module (see \cite{FHL93,HLZ10II, Li99b}).
	Let $V$ be a vertex operator algebra, and let $(W, Y_W)$ be a  lower-truncated generalized  $V$-module.
	Let 
	\[
	W^{\prime} = \bigoplus_{\lambda \in \mathbb{C}} \big( W_{(\lambda)} \big)^*
	\]
	be the restricted dual space of $W$. 
	Define  the vertex operator $Y_{W^{\prime}}$ on $W^{\prime}$ by
	\begin{align} \label{dual-id}
	\langle Y_{W^{\prime}}(v, x)w^{\prime}, \  w  \rangle 
    =\big\langle w^{\prime}, \ Y_W \big( e^{xL(1)}(-x^{-2})^{L(0)}v, x^{-1} \big)w \big \rangle,
	\end{align}
	for all $v \in V, w \in W$, and $ w^{\prime} \in W^{\prime}$. 	
    Throughout this paper, the pairing $\langle w', w \rangle$ denotes the evaluation
    $w^{\prime}(w)$ for $w^{\prime} \in W^{\prime}$ and $w \in W$. 
    Note that the lower-truncated property of $W$ guarantees that $Y_{W^{\prime}}(v, x)w^{\prime}$ contains only finitely many negative powers of $x$. 
    Under these definitions, $(W^{\prime}, Y_{W^{\prime}})$ forms a weak $V$-module.

Now assume further that $W$ is a restricted $V$-module. For any $\lambda \in \C$, by the finite-dimensionality of $W_{(\lambda)}$, there exists
$K \in \N$ such that $(L(0)-\lambda)^KW_{(\lambda)}=0$.
Then for any $w^{\prime} \in (W_{(\lambda)})^*$, we have
\[
 \big \langle (L(0)-\lambda)^K w^{\prime}, \ w  \big \rangle=\big \langle w^{\prime}, \  (L(0)-\lambda)^K w  \big \rangle=0
\]
for all $w \in W$. 
It follows that $(W', Y_{W'})$ is again a restricted $V$-module,
whose weight spaces coincide with the duals of those of $W$. That is,
\[
(W')_{(\lambda)} = \big(W_{(\lambda)}\big)^*
\]
for every weight $\lambda$. In particular, $\operatorname{wt}(W) = \operatorname{wt}(W')$.
Furthermore, the canonical evaluation map gives a $V$-module isomorphism $W \cong W''$. 

	\vspace{0.5\baselineskip}

Since the canonical isomorphism between a restricted $V$-module $W$ and its double restricted dual will be invoked repeatedly in later arguments, we adopt the following convention for simplicity:

\vspace{0.2\baselineskip} 

{\bf Convention:} 
Let $W$ be a graded vector space with finite-dimensional homogeneous components, and let $\theta_W: W \to W^{\prime \prime}$ 
denote the natural isomorphism onto its double restricted dual. 
From now on, we identify $W$ with its double restricted dual $W^{\prime \prime}$ via $\theta_W$ and omit writing  $\theta_W$ explicitly.
In particular, for $w\in W$ and $w''\in W''$, the equality $w = w''$ signifies
$\langle w^\prime, w \rangle = \langle w'', w^\prime \rangle$
for all $w^\prime \in W'$. 
	
\section{Logarithmic Intertwining Operators} \label{sec3}

The definition of (logarithmic) intertwining operators and their basic properties are standard, 
and can be found in \cite{CKM24,FHL93, HLZ10II, M02, M08}.
	
\subsection{ Logarithmic Intertwining Operators and Tensor Products}
    
\begin{definition}\label{def:Intertwining}
	Let $W^1, W^2$ and $W^3$ be weak $V$-modules.
	A {\bf logarithmic  intertwining operator} of type $\binom{W^3}{W^1 \ W^2}$ is a linear map
	\begin{align*}
		\ \	\mathcal{Y}(- , x)-: & \ W^1 \otimes W^2  \  \to  \ W^3 [\log x]  \{x\} \\
		&w^1 \otimes w^2 \ \to \  \Y(w^1, x)w^2=\sum_{n \in \C}  \sum_{ k=0}^{p_n} (w^1)^\Y_{n;k}w^2x^{-n-1}(\log x)^k \in  W_3 [\log x]  \{x\},
	\end{align*}
	where $p_n \in \N$, satisfying the following conditions:

    \begin{enumerate}[{(1)}]
    
		\item Truncation property: 
		For any $w^1 \in W^1, w^2 \in W^2$, and $h \in \C,$ 
        $(w^1)^{\Y}_{h+n;k}w^2=0 $ for all sufficiently large integers $n$, independently over $k$;
			
		\item Jacobi identity: For all $v \in V$, $w^1 \in W^1$, $w^2\in W^2$,
		\begin{align*} \label{Jacobi}
			x_0^{-1} \delta\!\left(\frac{x_1-x_2}{x_0}\right) & Y_{W^3}(v,x_1) \mathcal{Y}(w^1,x_2)w^2 
			- x_0^{-1} \delta\!\left(\frac{-x_2+x_1}{x_0}\right) \mathcal{Y}(w^1,x_2) Y_{W^2}(v,x_1)w^2 \notag \\
			&= x_2^{-1} \delta\!\left(\frac{x_1-x_0}{x_2}\right) \mathcal{Y}\big(Y_{W^1}(v,x_0)w^1, x_2\big)w^2;
		\end{align*}
				  
		\item $L(-1)$-derivative property: For all $w^1 \in W^1$, 
		\begin{equation*} \label{L(-1)-der}
			\Y \big(L(-1)w^1, x \big)=\frac{d}{d x}\Y(w^1, x).    
		\end{equation*}
	\end{enumerate}
        
	For simplicity, we write $w^1_{n;k}w^2$ instead of $(w^1)^\Y_{n;k}w^2$ when no confusion arises. 
	From now on, 'intertwining operator' means 'logarithmic intertwining operator' unless otherwise specified.
	The set of all intertwining operators of type $\binom{W^3}{W^1 \ W^2}$ forms a vector space, 
    denoted by  $\mathcal{I} \binom{W^3}{W^1 \ W^2}$. 
	Its dimension is called the {\bf fusion rule} for $W^1, W^2$ and $W^3$, and is denoted by $N(W^1, W^2; W^3).$	
\end{definition}

Let $W^1, W^2$ and $W^3$ be weak $V$-modules.
Let $\Y$ be an intertwining operator of type $\binom{W^3}{W^1, W^2}$. 
Its image, $ \operatorname{Im} \Y \subseteq W^3$, is defined as the subspace spanned by all coefficients of $\Y(w^1, x)w^2$ for all $w^1 \in W^1$ and $w^2 \in W^2$.
Note that $\operatorname{Im} \Y$ is a $V$-submodule of $W^3$.
If $\operatorname{Im}\Y =W^3$, then $\Y$ is called a {\bf surjective intertwining operator}.
We say that \(\mathcal{Y}\) is {\bf injective in the first argument} if \(\mathcal{Y}(w^1,x)=0\) implies \(w^1=0\).
    
\begin{definition}
	Let $W^1$ and $W^2$ be restricted  $V$-modules. 
	A \textbf{tensor product} of $W^1$ and $W^2$ in the category of  restricted  $V$-modules 
	is a pair $\big(W^1\boxtimes W^2, \mathcal{Y}^{\boxtimes} \big)$, where $W^1\boxtimes W^2$ is a restricted  
	$V$-module and $\mathcal{Y}^{\boxtimes}$ is an intertwining operator of type 
	$\binom{W^1\boxtimes W^2}{W^1 \ W^2}$, satisfying the following universal property: 
	for any restricted  $V$-module $W$ and any intertwining operator $\mathcal{Y}$ of type $\binom{W}{W^1 \ W^2}$, there 
	exists a unique $V$-module homomorphism $\phi: W^1\boxtimes W^2 \to W$ such that
	$\mathcal{Y} = \phi \circ \mathcal{Y}^{\boxtimes}.$
\end{definition}

If the tensor product $W^1\boxtimes W^2$ exists, then $\Y^{\boxtimes}$ is a surjective intertwining operator.
Moreover, the universal property gives an isomorphism between vector spaces
\[
\operatorname{Hom}_{V}(W_1 \boxtimes W_2, W) \longrightarrow 
\mathcal{I}\binom{W}{W_1 \; W_2}, \quad
\phi \longmapsto \phi \circ \mathcal{Y}^{\boxtimes},
\]
for any restricted $V$-module $W$.

\vspace{0.5\baselineskip} 

We have the following result \cite{M14, H25C1, YZ26}: 
\begin{theorem} \label{C1-tensor}
	Let $W^1$ and $W^2$ be $C_1$-cofinite restricted $V$-modules. 
    Then the tensor product $W^1 \boxtimes W^2$ exists and is also a $C_1$-cofinite restricted $V$-module.
\end{theorem}

\subsection{ The definitions of $\Omega_r(\Y)$ and $A_r(\Y)$  }
    
	We next define the intertwining operators $\Omega_r(\Y)$ and $A_r(\Y)$ for every $r \in \Z$ (cf.  \cite{CKM24, FHL93,HLZ10II}).
	Let $W$ be any vector space,  and consider a formal series
	\[f(x)=\sum_{n \in \C} \sum_{ k=0}^{p_n}   w_{n;k} x^{-n-1}(\log x)^k \in W[\log x] \{x\},
	\] 
    where $p_n \in \N$.
	In the subsequent definitions of $\Omega_r(\Y)$ and $A_r(\Y)$, we adopt the following conventions for the expressions
    $f \big( e^{ (2r+1)\pi \mathrm{i}} x \big) $ ($r \in \Z$) and $f(x^{-1})$:
	\begin{align*}
	f \big( e^{ (2r+1)\pi \mathrm{i}} x \big) 
    & = \sum_{ n \in \C}\sum_{k=0}^{p_n}  e^{(-n-1)  (2r+1)\pi \mathrm{i}}  w_{n;k}  x^{-n-1} \big(   (2r+1)\pi \mathrm{i}  +   \log x \big)^k \\
    & = \sum_{ n \in \C}\sum_{k=0}^{p_n} \sum_{j=0}^k  \binom{k}{j}
    e^{(-n-1)  (2r+1)\pi \mathrm{i}} \big((2r+1)\pi \mathrm{i} \big)^{j} w_{n;k}  x^{-n-1} (\log x)^{k-j}.
	\end{align*}
	and
	\[
	f(x^{-1})= \sum_{n \in \C}  \sum_{ k=0}^{p_n} (-1)^k w_{n;k} x^{n+1}(\log x)^k.
	\]

	Let $W^1, W^2$ and $W^3$ be weak $V$-modules, and
	let $\Y$ be an intertwining operator of type $\binom{W^3}{W^1, W^2}$. 
	For any $r \in \Z$, define a linear map 
	\[
	\Omega_r(\Y): W^1 \otimes W^2 \to W^3 [\log x] \{x\}
	\] 
	by 
	\[
	\Omega_r(\Y) \big(w^2, x \big)w^1 = e^{xL(-1)} \Y \big(w^1, e^{ (2r+1)\pi \mathrm{i}}x \big)w^2, 
	\]
	for all $ w^1 \in W^1, w^2 \in W^2.$
	Then $\Omega_r(\Y)$ is an intertwining operator of type  $\binom{W^3}{W^2, W^1}$. 
	Moreover,  we have
	\[
	\Omega_{-r-1}(\Omega_r(\Y))=	\Omega_{r}(\Omega_{-r-1}(\Y)) =\Y.
	\]

	Let $W^1$ be a generalized $V$-module, and $ W^2$ and $W^3$ be restricted $V$-modules. 
	Suppose $\Y$ is an intertwining operator of type $\binom{W^3}{W^1, W^2}$ and $r \in \Z$. 
	Define a linear map
	\[
	A_r(\Y): W^1 \otimes (W^3)^\prime \to (W^2)^\prime[\log x] \{x\}
	\]
	by
	\[
	\left<A_r(\Y)(w^1, x)w_3^{\prime}, \  w^2 \right >
    =\left<w_3^{\prime}, \ \Y \big( e^{xL(1)}  (e^{(2r+1) \pi \mathrm{i} } x^{-2} \big)^{L(0)} w^1  , x^{-1} \big)w^2 \right>
	\]
	for all $w^1 \in W^1, w^2 \in W^2$ and $w_3^{\prime}\in (W^3)^\prime$.
	Then $A_r(\Y)$ is an intertwining operator of type  
    $\binom{(W^2)^\prime}{W^1, \ (W^3)^\prime}$.
    Moreover, we have
	\[
	A_{-r-1}\big( A_r(\Y) \big)=	A_{r}\big( A_{-r-1}(\Y) \big) =\Y.
	\]
	
We note that the lower-truncated condition of $W^2$ implies that
the powers of $x$ appearing in $A_r(\mathcal{Y})(w^1, x)w_3'$ are lower-truncated.
In addition, since all weight spaces of $W^2$ and $W^3$ are finite-dimensional,
for any $n\in\mathbb{C}$, the coefficient of $x^{-n-1}$ in $A_r(\mathcal{Y})(w^1, x)w_3'$
lies in $(W^2)'[\log x]$.

\subsection{Some Lemmas}

In this subsection, we collect several lemmas on intertwining operators that will be used later.

\begin{lemma}\label{gen-submodule-lemma}
Let $W^1, W^2, W^3$ be weak $V$-modules, and let $\Y \in \mathcal{I}\binom{W^3}{W^1 \ W^2} $.
Let $w \in W^1$, and let $W$ be the submodule of $W^1$ generated by $w$.
Suppose $M$ is a submodule of $W^3$ such that
\[
\Y(w, x)w^2 \in M[\log x]\{x\},  \; \; \forall \; w^2 \in W^2.
\]
Then 
\[
\Y(w^1, x)w^2 \in M[\log x]\{x\},  \; \;  \forall \;  w^1 \in W, \;  w^2 \in W^2.
\]
In particular, taking $M = 0$, the annihilator
\[
K=\{ w^1 \in W^1 \ | \ \Y(w^1, x)=0 \} \subset W^1.
\]
is a $V$-submodule of $W^1$.

\end{lemma}

\begin{proof}
For any $u \in V$, $n\in\mathbb{Z}$, and $w^2 \in W^2$, the Jacobi identity for $\Y$ gives
\begin{align*}
\Y(u_n w,x)w^2
&= \operatorname{Res}_{x_1}\Big\{(x_1-x)^nY_{W^3}(u,x_1)\Y(w,x)w^2 \\
&\qquad\qquad -(-x+x_1)^n\Y(w,x)Y_{W^2}(u,x_1)w^2\Big\} \\
&\in M[\log x]\{x\}.
\end{align*}
Since $W$ is generated by $w$, the claim follows.

\end{proof}

\begin{lemma} \label{Inj-Sur}
Suppose $r,s\in\mathbb{Z}$.
Let $W^1,W^2,W^3$ be restricted $V$-modules, and let $\Y$ be an intertwining operator of type $\binom{W^3}{W^1 \; W^2}$,
Then the following statements hold:
\begin{enumerate}[(1)]
\item Let $w \in W^1$. Then
\[
\bigl\langle A_r \Omega_s \mathcal{Y} (w^2, x) w_3',\; w \bigr\rangle =0, 
\; \;  \; \; \forall \; w^2 \in W^2, \;  w_3^\prime \in (W^3)^\prime, 
\]
if and only if
 \[
\bigl\langle w_3', \; \mathcal{Y} \big(w,  x \big) w^2 \bigr\rangle =0, 
\,  \,  \,  \,  \forall \, w^2 \in W^2, \;  w_3^\prime \in (W^3)^\prime.
\]

\item The intertwining operator $\mathcal{Y}(-, x)-$ is injective in its first argument if and only if $ A_r \Omega_s \mathcal{Y}$ is surjective.
\end{enumerate}
\end{lemma}

\begin{proof}
(1)  ($  \Rightarrow $)  \,  By the definitions of $\Omega_s\Y$ and $A_r\Y$, we have
\begin{align*}
0 &= \bigl\langle A_r \Omega_s \mathcal{Y} (w^2, x) w_3', \, w \bigr\rangle \\
&= \Bigl\langle  w_3',\; e^{x^{-1}L(-1) } \mathcal{Y} \big(w, \, e^{(2s+1)\pi \mathrm{i}} x^{-1} \big)
e^{xL(1)} \big( e^{(2r+1)\pi \mathrm{i}} x^{-2} \big)^{L(0)}  w^2 \Bigr\rangle \\
&= \Bigl\langle  e^{x^{-1}L(1) }  w_3',\; \mathcal{Y} \big(w, \, e^{(2s+1)\pi \mathrm{i}} x^{-1} \big)
e^{xL(1)} \big( e^{(2r+1)\pi \mathrm{i}} x^{-2} \big)^{L(0)}  w^2 \Bigr\rangle
\end{align*}
for all $w^2\in W^2$ and $w_3'\in (W^3)'$.
Making the substitutions
\[
w^2\mapsto \big(e^{(2r+1)\pi \mathrm{i}} x^{-2} \big)^{-L(0)} e^{-xL(1)} w^2 \quad   \text{and} \quad
w_3'\mapsto e^{-x^{-1}L(1)} w_3'
\]
yields
\[
\bigl\langle w_3',\; \mathcal{Y} \big(w, \, e^{(2s+1)\pi \mathrm{i}} x^{-1} \big)w^2 \bigr\rangle =0
\]
for all $w^2\in W^2$ and $w_3'\in (W^3)'$. 
Finally, substituting $x$ by $e^{(2r+1)\pi \mathrm{i}} x^{-1} $, we arrive at
\[
\bigl\langle w_3',\; \mathcal{Y} \big(w,  x \big) w^2 \bigr\rangle =0
\]
for all $w^2\in W^2$ and $w_3'\in (W^3)'$.

($ \Leftarrow $) \,  Since all substitutions above are invertible,  reversing the steps proves the reverse implication.

Hence, the two vanishing conditions are equivalent, finishing the proof of (1).

\vspace{0.5\baselineskip}

(2) The assertion follows immediately from part (1).

\end{proof}

\begin{lemma}
 
Let $ W_i, W^i \; ( i=1, 2, 3) $  be weak $V$-modules,
and let
\[
\sigma_1: W^1 \to W_1, \ \sigma_2: W^2 \to W_2  \ \ \text{and} \ \ \sigma_3: W_3 \to W^3
\]
be \(V\)-module homomorphisms.  
Let \(\mathcal{Y}(-,x)-\) be an intertwining operator of type \(\binom{W_3}{W_1\,W_2}\).  
Define the operator $[ \sigma_1,\sigma_2, \sigma_3]\circ \mathcal{Y} $ by
\[
 \big([ \sigma_1,\sigma_2, \sigma_3]\circ \mathcal{Y} \big)(w^1, x)w^2 = 
 \sigma_3 \Big(  \mathcal{Y}\bigl(\sigma_1 ( w^1), x\bigr)  \sigma_2( w^2) \Big)
\]
for all $w^1 \in W^1$ and $w^2 \in W^2$.
Then $ [ \sigma_1,\sigma_2, \sigma_3]\circ \mathcal{Y} $  is an intertwining operator of type \(\binom{W^3}{W^1\,W^2}\).
Moreover,  if $W_1$ is a generalized $V$-module and $W_2, W_3$ are  restricted $V$-modules,
then for any $r \in \N$, we have 
\[
\Omega_r \big([\sigma_1,\sigma_2, \sigma_3]\circ\mathcal{Y}\big) = [\sigma_2,\sigma_1, \sigma_3]\circ ( \Omega_r \mathcal{Y}) 
\]
and
\[
A_r \big([\sigma_1,\sigma_2, \sigma_3]\circ\mathcal{Y}\big) = [\sigma_1, \sigma_3^*, \sigma_2^* ]\circ ( A_r \mathcal{Y}). 
\]
\end{lemma}

\begin{proof}
It is straightforward to see that $[ \sigma_1,\sigma_2, \sigma_3]\circ \mathcal{Y}$ is an intertwining operator of type \(\binom{W^3}{W^1\,W^2}\).
For the two identities, we only prove the second one. The first one can be proved similarly.
By definition,  for all $w^1 \in W^1, w^2 \in W^2$ and $w_3^\prime \in (W^3)^\prime$, we have
\begin{align*}
& \left \langle A_r \big([\sigma_1,\sigma_2, \sigma_3]\circ\mathcal{Y}\big)(w^1, x) w_3^\prime, \ w^2 \right \rangle \\
&=  \left \langle w_3^\prime, \;   \big([\sigma_1,\sigma_2, \sigma_3]\circ\mathcal{Y}\big) \big( e^{xL(1)} 
\big( e^{(2r+1)\pi \mathrm{i}} x^{-2} \big)^{L(0)}w^1,  x^{-1}  \big)w^2    \right \rangle \\
&=  \left \langle w_3^\prime, \;   \sigma_3 \mathcal{Y} \big( \sigma_1 e^{xL(1)} 
\big( e^{(2r+1)\pi \mathrm{i}} x^{-2} \big)^{L(0)}w^1,  x^{-1}  \big) \sigma_2w^2    \right \rangle \\
&=  \left \langle \sigma_3^* w_3^\prime, \;    \mathcal{Y} \big( e^{xL(1)} \big( e^{(2r+1)\pi \mathrm{i}} x^{-2} \big)^{L(0)}  \sigma_1 w^1,  x^{-1}  \big) \sigma_2w^2 
   \right \rangle\\
&=  \left \langle (A_r \mathcal{Y})(\sigma_1w^1, x) \sigma_3^* w_3^\prime, \; \sigma_2w^2 \right \rangle \\
&=  \left \langle \sigma_2^* ( A_r \mathcal{Y}) (\sigma_1w^1, x) \sigma_3^* w_3^\prime, \;  w^2 \right \rangle\\
&=  \left \langle  \big( [\sigma_1, \sigma_3^*, \sigma_2^* ]\circ ( A_r \mathcal{Y})  \big) (w^1, x )w_3^\prime, w^2   \right \rangle.
\end{align*}
Consequently,
\[
A_r \big([\sigma_1,\sigma_2, \sigma_3]\circ\mathcal{Y}\big) = [\sigma_1, \sigma_3^*, \sigma_2^* ]\circ ( A_r \mathcal{Y}),
\]
which completes the proof.

\end{proof}

\section{The Internal Hom  $\H(W^1, W^2)$ } \label{sec4}

\subsection{ Construction of the Internal Hom  $\H(W^1, W^2)$ }

In this subsection, for any two restricted $V$-modules $W^1$ and $W^2$,  we construct a generalized $V$-module $\H(W^1, W^2)$
along with a canonical intertwining operator $\Y^\H$ of type $ \binom{W^2}{\mathcal{H}(W^1, W^2)\ W^1}$. 
The resulting pair $\big( \H(W^1, W^2), \Y^\H  \big)$ satisfies a certain universal property.
We remark that, despite the difference in construction methods, our approach is inspired by the results in \cite{Li98}.

\begin{definition}
    
Let $V$ be a vertex operator algebra, and let $W^1, W^2$ be restricted $V$-modules.
We define $\mathcal{H}(W^1,W^2)$ to be the subspace of
$\operatorname{Hom}_{\mathbb{C}}\big(W^1, W^2[\log t]\{ t \}\big)$
spanned by all operators of the form $\mathcal{Y}_A(a,t)$,
where $A$ ranges over all restricted $V$-modules,
$\mathcal{Y}_A$ ranges over all intertwining operators of type $\binom{W^2}{A\;W^1}$,
and $a$ ranges over $A$. 
We say that $ \H(W^1,W^2)$ is the {\bf internal Hom} from $ W^1$ to $ W^2$.

\end{definition}

\begin{remark}
Although the internal Hom is a concept in tensor categories, we adopt this terminology here because, under suitable assumptions, 
Theorem \ref{inter-hom-th} identifies $\mathcal{H}(W^1, W^2)$ with the internal Hom 
in the tensor category of restricted $V$-modules.

\end{remark}

Let 
\[
\sum_{i=1}^n\lambda_i \Y_{A^i}(a^i, t) \in  \mathcal{H}(W^1,W^2),
\]
where for each index $i$, $\lambda_i \in \C$,  $A^i$ is a restricted $V$-module, $a^i \in A^i$, and $\Y_{A^i} \in \mathcal{I}\binom{W^2}{A^i \ W^1}$.
Set
\[
A= A^1 \oplus \cdots \oplus A^n, \ \ a=( \lambda_1 a^1, \cdots, \lambda_n a^n) \in A ,  \ \ \ \text{and} \ \ \ \Y_A= \Y_{A^1} \oplus \cdots \oplus \Y_{A^n}. 
\]
Then $A$ is also a restricted $V$-module, and $\Y_A$ is an intertwining operator of type $\binom{W^2}{A\; W^1}$.
Moreover,
\[
\sum_{i=1}^n\lambda_i \Y_{A^i}(a^i, t) =\Y_A(a, t).
\]
It follows that every element of $\H(W^1, W^2)$ can be written in the form $\Y_A(a, t)$ for some restricted $V$-module $A$,
some $a \in A$, and some intertwining operator $ \Y_A \in  \mathcal{I}\binom{W^2}{A \ W^1}$.

We emphasize that there might exist non-isomorphic restricted $V$-modules $A$ and $B$, 
elements $a \in A$ and $b \in B$, and intertwining operators 
\[
\Y_A \in \mathcal{I}\binom{W^2}{A\ W^1}, \ \  \Y_B \in \mathcal{I}\binom{W^2}{B\ W^1}
\]
such that $\Y_A(a, t)$ and $\Y_B(b, t)$ represent the same element in $\mathcal{H}(W^1, W^2)$.

\vspace{0.5\baselineskip} 

From now on, to simplify notation, 
writing $\mathcal{Y}_A(a, t) \in \mathcal{H}(W^1,W^2)$ will imply that $A$ is a restricted $V$-module, $a \in A$, 
and $\mathcal{Y}_A$ is an intertwining operator of type $\binom{W^2}{A\;W^1}$.

    For $u \in V, n \in \Z$, and $\Y_{A}(a, t) \in  \mathcal{H}(W^1,W^2)$, 
    we define
    \[
    u^{\mathcal{H}}_n  \Y_{A}(a, t) =  \Y_A(u_na, t) \in \mathcal{H}(W^1,W^2).
    \]
	Suppose that $\Y_A(a, t) $ and $\Y_B(b, t) $ represent the same element in $\H(W^1, W^2)$.
	Then we have
	\begin{align*} \label{unY}
		u^{\mathcal{H}}_n  \Y_{A}(a, t) 
		  & = \text{Res}_{x} \Bigl\{ (x-t)^n Y_{W^2}(u, x) \Y_{A}(a, t) 
		- (-t+x)^n \Y_{A}(a, t)  Y_{W^1}(u, x) \Bigr\} \notag \\ 
		&= \text{Res}_{x} \Bigl\{ (x-t)^n Y_{W^2}(u, x)   \Y_{B}(b, t)
		- (-t+x)^n \Y_{B}(b, t) Y_{W^1}(u, x) \Bigr\} \notag \\
		& =\Y_{B}(u_nb, t) = u^{\mathcal{H}}_n  \Y_{B}(b, t).
	\end{align*}
	Hence, the action of $u_n^{\mathcal{H}}$ is well-defined.
	
	For any $u \in V$, set 
	$$Y_{\mathcal{H}}(u, x)= \sum_{n \in \Z } u^{\mathcal{H}}_n  x^{-n-1}.$$ 
	Then, for any $\Y_A(a, t) \in \H(W^1, W^2)$, we have 
    \begin{align*}
    Y_{\mathcal{H}}(u, x)  \Y_A(a,t) &= \sum_{n \in \Z}u_n^{\mathcal{H}} \Y_A(a, t)  x^{-n-1}= \sum_{n \in \Z } \Y_A(u_na ,t)  x^{-n-1}\\
     &= \Y_A \big(Y_A(u, x)a, t \big) \in \H(W^1, W^2)((x)).
    \end{align*}
	Following \cite{Li98}, we further define a linear map
	\[
	\Y^{\H}(-,x): \H(W^1,W^2) \otimes W^1 \to W^2 [\log x] \{x\}  
	\]
	by
	\[
	\Y^{\H} \big(\Y_A(a, t), x \big)w^1=\Y_A(a, x)w^1.
	\]
	Note that $\Y^{\H}$ is likewise well-defined. 
    Moreover, $\Y^{\H}(-,x)-$ is injective in its first argument, a property that we will use repeatedly in what follows.

\vspace{0.5\baselineskip} 

Despite the difference in construction, parts (1) and (2) of Theorem \ref{interHom1} below may be viewed as counterparts of Theorems 4.6 and 4.7 in \cite{Li98}, respectively.
% These latter results will be subsumed by Theorem \ref{Th-G} later.

\begin{theorem} \label{interHom1}
Let $V$ be a vertex operator algebra, and let $W^1, W^2$ be restricted $V$-modules. 
Then the following statements hold:

\begin{enumerate}[(1)]
  
    \item  With the above definitions,  $\mathcal{H}(W_1,W_2)$ is a generalized $V$-module, and 
	   $\mathcal{Y}^{\mathcal{H}}(-,x) $ is an intertwining operator of type 
	   $ \binom{W^2}{\mathcal{H}(W^1, W^2)\ W^1}$; 

    \item The pair \(\big(\mathcal{H}(W^1, W^2),\mathcal{Y}^{\mathcal{H}}\big)\) satisfies the following {\em universal property}: For any generalized $V$-module $W$ that is sum of its restricted submodules and any intertwining operator $\Y_W$ of type $\binom{W^2}{W \ W^1}$, there exists a unique $V$-module homomorphism
		\[
        \varphi_W \colon W \to \mathcal{H}(W^1,W^2)
        \]
        such that
		\[
		\mathcal{Y}^{\mathcal{H}}\big(\varphi_W(w),x\big) = \mathcal{Y}_W(w,x)
		\]
		for all \(w\in W\). Consequently, the map
		\vspace{-0.5\baselineskip}
		\[
		\operatorname{Hom}_{V} \big(W ,  \mathcal{H}(W^1, W^2) \big) \longrightarrow 
		\mathcal{I}\binom{W^2}{W \; W^1}, \quad
		\varphi \longmapsto \mathcal{Y}^{\mathcal{H}}(\varphi(-), x), 
		\vspace{-0.5\baselineskip}
		\]
		is a linear isomorphism;

  \item  $ \H(W^1, W^2) $ is the sum of its restricted \(V\)-submodules.

\end{enumerate}

\end{theorem}

\begin{proof}

	(1) Let \((A,Y_A)\) be a restricted \(V\)-module, \(a\in A\), and 
    $\mathcal{Y}_A \in \mathcal{I}\binom{W^2}{A\ W^1}$.
	By the definition of \(\mathcal{Y}^{\mathcal{H}}\) and the Jacobi identity for the \(V\)-module \(A\), we have for all \(u,v\in V\),
	\begin{align*}
		& x_0^{-1} \delta\left(\frac{x_1-x_2}{x_0}\right) Y_{\mathcal{H}}(u,x_1) Y_{\mathcal{H}}(v,x_2) \mathcal{Y}_A(a,t) \\
		& - x_0^{-1} \delta\left(\frac{-x_2+x_1}{x_0}\right) Y_{\mathcal{H}}(v,x_2) Y_{\mathcal{H}}(u,x_1) \mathcal{Y}_A(a,t) \\
		& = x_0^{-1} \delta\left(\frac{x_1-x_2}{x_0}\right) \mathcal{Y}_A\Big(Y_A(u,x_1)Y_A(v,x_2)a,t\Big) \\
		& - x_0^{-1} \delta\left(\frac{-x_2+x_1}{x_0}\right) \mathcal{Y}_A\Big(Y_A(v,x_2)Y_A(u,x_1)a,t\Big) \\
		& = x_2^{-1} \delta\left(\frac{x_1-x_0}{x_2}\right) \mathcal{Y}_A\Big(Y_A\big (Y(u,x_0)v,x_2 \big)a,t\Big) \\
		& = x_2^{-1} \delta\left(\frac{x_1-x_0}{x_2}\right) Y_{\mathcal{H}}\big(Y(u,x_0)v,x_2\big) \mathcal{Y}_A(a,t).
	\end{align*}
	Thus, the Jacobi identity holds for \(Y_{\mathcal{H}}\). The vacuum property and the lower truncation property are straightforward to check.
    Thus, \(\mathcal{H}(W^1,W^2)\) is a weak \(V\)-module. The $L(0)$ action on $\H(W^1, W^2)$ immediately implies that it is a generalized $V$-module.
	
	 By the definitions of \(\mathcal{Y}^{\mathcal{H}}\) and \( Y_{\H}\), together with the Jacobi identity for \(\mathcal{Y}_A\), we obtain
	\begin{align*}
		& x_0^{-1} \delta\left(\frac{x_1-x_2}{x_0}\right) Y_{W^2}(u,x_1) \mathcal{Y}^{\mathcal{H}}\big(\mathcal{Y}_A(a,t),x_2\big) \\
		& - x_0^{-1} \delta\left(\frac{-x_2+x_1}{x_0}\right) \mathcal{Y}^{\mathcal{H}}\big(\mathcal{Y}_A(a,t),x_2\big) Y_{W^1}(u,x_1) \\
		& = x_0^{-1} \delta\left(\frac{x_1-x_2}{x_0}\right) Y_{W^2}(u,x_1)\mathcal{Y}_A(a,x_2) 
		- x_0^{-1} \delta\left(\frac{-x_2+x_1}{x_0}\right) \mathcal{Y}_A(a,x_2)Y_{W^1}(u,x_1) \\
		& = x_2^{-1} \delta\left(\frac{x_1-x_0}{x_2}\right) \mathcal{Y}_A\big(Y_A(u,x_0)a,x_2\big) \\
        & = x_2^{-1} \delta\left(\frac{x_1-x_0}{x_2}\right) Y_{\H}(u,x_0)\Y_A(a, x_2)\\
		& = x_2^{-1} \delta\left(\frac{x_1-x_0}{x_2}\right) \mathcal{Y}^{\mathcal{H}}\Big(Y_{\mathcal{H}}(u,x_0)\mathcal{Y}_A(a,t),x_2 \Big).
	\end{align*}
	The lower-truncation property and \(L(-1)\)-derivation property of \(\mathcal{Y}^{\mathcal{H}}\) are trivially satisfied.
	Hence, \(\mathcal{Y}^{\mathcal{H}}\) is an intertwining operator.

	\vspace{0.5\baselineskip}

(2) By the definition of $W$, for any $ w \in W$, there exists a restricted submodule $A \subseteq W$ with $w \in A$. 
Hence
\[
\mathcal{Y}_W(w, t) = \mathcal{Y}_W|_A(w,t) \in \mathcal{H}(W^1,W^2).
\]
We then define a linear map
\[
\varphi_W \colon W \to \mathcal{H}(W^1,W^2)
\]
by
\[
\varphi_W(w) = \mathcal{Y}_W(w,t) \in \mathcal{H}(W^1,W^2).
\]
%%Clearly, the definition of $\varphi_W$ is independent of the choice of the restricted $V$-module $A$. 
From the definition of the actions $u^{\mathcal{H}}_n$, it follows that $\varphi_W$ is a $V$-module homomorphism.

Moreover, by the definition of $\mathcal{Y}^{\mathcal{H}}$, we have
\[
\mathcal{Y}^{\mathcal{H}}\big(\varphi_W(w),x\big)
= \mathcal{Y}^{\mathcal{H}}\big(\mathcal{Y}_W(w,t),x\big)
=\mathcal{Y}_W(w,x)
\]
for any $w \in W$. Since $\mathcal{Y}^{\mathcal{H}}(-,x)$ is injective in its first argument, such a homomorphism $\varphi_W$ is unique.

	\vspace{0.5\baselineskip} 

    (3) Take an arbitrary $\Y_A(a, t) \in \H(W^1, W^2),$  
    where $A$ is a restricted $V$-module, $a \in A$, and $\Y$ is an intertwining operator of type $ \binom{W^2}{A \ W^1}.$
   Then, by (2) we have
   \[
   \Y_A(a, t) = \varphi_A(a) \in \operatorname{Im}\varphi_A. 
   \] 
   Since \(\operatorname{Im}\varphi_A\) is a homomorphic image of the restricted \(V\)-module \(A\), it is itself a restricted \(V\)-module. Consequently, \(\mathcal{H}(W^1,W^2)\) coincides with the sum of its restricted \(V\)-submodules.
	This completes the proof.
\end{proof}

\begin{remark}

\begin{enumerate}[(1)]

\item   Although $\mathcal{H}(W^1, W^2)$ is the sum of its restricted $V$-submodules, 
this does not necessarily imply that it is itself a restricted $V$-module.

\item In the definition of \(\mathcal{H}(W^1,W^2)\), $A$ is taken to be a restricted $V$-module.
However, one may instead allow $A$ to be a generalized $V$-module, or even an arbitrary $V$-module. 
The resulting construction gives a larger weak $V$-module that contains $\mathcal{H}(W^1, W^2)$ 
and still satisfies the corresponding universal property.

\end{enumerate}
	
\end{remark}
Let $W^1, W^2$ be restricted $V$-modules.
We define a category $\mathcal{D}(W^1, W^2)$ whose objects are pairs $(W, \mathcal{Y})$, 
where $W$ is a generalized $V$-module that is the sum of its restricted submodules, 
and $\mathcal{Y}$ is an intertwining operator of type $\binom{W^2}{W\; W^1}$. 
A morphism $\varphi$ between $(W, \mathcal{Y})$ and $(\tilde{W}, \tilde{ \Y} )$ is $V$-module map from 
$\tilde{W}$ to $W$ such that $\tilde{\Y}(-, x)=\Y\big(\varphi(-) ,  x\big)$.
Then, the pair $\big( \H(W^1, W^2), \Y^\H\big)$ is an object in $ \mathcal{D}(W^1, W^2)$.
Theorem \ref{interHom1} (2) shows that $\left(\H(W^1, W^2), \mathcal{Y}^{\H}\right)$ serves as an initial object in $\mathcal{D}(W^1, W^2)$.
Consequently, this object is unique up to isomorphism in \(\mathcal{D}(W^1, W^2)\), 
and it is appropriate to use the term 'universal property' in Theorem \ref{interHom1} (2). 

The following theorem shows that, in the category $ \mathcal{D}(W^1, W^2)$, 
in order to verify that a pair $  (M, \Y^*)$ satisfies the universal property, 
it suffices to check the condition for every pair $(A, \mathcal{Y})$ where $A$ is restricted.

\begin{theorem} \label{Universal-H}

Let $(M, \Y^*) \in \mathcal{D}(W^1, W^2).$
Then the pair $( M, \Y^*) $ is an initial object in $\mathcal{D}(W^1, W^2)$ if and only if 
$(M, \Y^*) $ satisfies the following property: 
for any restricted $V$-module $A$ and any intertwining operator $\Y_A$ of type $\binom{W^2}{A \ W^1}$,
there exists a unique $V$-module homomorphism $\phi_A: A \to M$ such that
\[
\Y^* \big( \phi_A(a), x\big) = \mathcal{Y}_A(a,x)
\]
for all \(a \in A\).
\end{theorem}

\begin{proof}

Suppose the pair $(M, \mathcal{Y}^*)$ satisfies the above property. We first show that $\mathcal{Y}^*$ is injective in its first argument.
Assume for contradiction that $\mathcal{Y}^*(m, x) = 0$ for some nonzero element $m \in M$. Let
\[
K\overset{\mathrm{def}} { =} \{ k \in M \mid \mathcal{Y}^*(k, x) = 0 \} \subseteq M.
\]
Since $m \in K$, $K$ is a nonzero submodule of $M$. Choose a restricted $V$-submodule $A \subseteq K$ containing $m$. Both the inclusion map $i_A \colon A \to M$ and the zero map $0_A \colon A \to M$ satisfy
\[
\mathcal{Y}^*\big(i_A(a), x\big) = \mathcal{Y}^*\big(0_A(a), x\big) = \mathcal{Y}^*\big|_A(a, x) = 0
\]
for all $a \in A$. By the uniqueness assertion in the defining property of $(M, \mathcal{Y}^*)$, we have $i_A = 0_A$. This contradicts the fact that $A$ is nonzero. Therefore, $\mathcal{Y}^*$ is injective in its first argument.

We now verify that $(M, \mathcal{Y}^*)$ is an initial object of $\mathcal{D}(W^1, W^2)$. Take an arbitrary object $(W, \mathcal{Y}) \in \mathcal{D}(W^1, W^2)$. We shall prove that there exists a unique $V$-module homomorphism $f_W \colon W \to M$ such that
\[
\mathcal{Y}^*\big(f_W(w), x\big) = \mathcal{Y}(w, x)
\]
for all $w \in W$. As $\mathcal{Y}^*$ is injective in its first argument, such a homomorphism $f_W$ is unique whenever it exists.

We now construct the desired map $f_W$. For any $w \in W$, since $W$ is the sum of its restricted submodules, there exists a restricted $V$-submodule $A_w \subseteq W$ with $w \in A_w$. Then the restriction $\mathcal{Y}\big|_{A_w}$ is an intertwining operator of type $\binom{W^2}{A_w \ W^1}$. By the property of $(M, \mathcal{Y}^*)$, there exists a unique $V$-module homomorphism
\[
\phi_{A_w} \colon A_w \to M
\]
satisfying
\[
\mathcal{Y}^*\big(\phi_{A_w}(a), x\big) = \mathcal{Y}\big|_{A_w}(a, x) = \mathcal{Y}(a, x)
\]
for all $a \in A_w$.

Define a map $f_W \colon W \to M$ by setting $f_W(w) = \phi_{A_w}(w)$ for each $w \in W$. We claim that $f_W$ is well-defined. Let $B_w \subseteq W$ be another restricted $V$-submodule containing $w$. Again by the given property, there exists a unique $V$-module homomorphism
\[
\phi_{B_w} \colon B_w \to M
\]
such that
\[
\mathcal{Y}^*\big(\phi_{B_w}(b), x\big) = \mathcal{Y}\big|_{B_w}(b, x) = \mathcal{Y}(b, x)
\]
for all $b \in B_w$. It follows that
\[
\mathcal{Y}^*\big(\phi_{A_w}(w), x\big) = \mathcal{Y}^*\big(\phi_{B_w}(w), x\big).
\]

Since $\mathcal{Y}^*$ is injective in its first argument, we obtain $\phi_{A_w}(w) = \phi_{B_w}(w)$. Hence, $f_W$ is well-defined. Moreover, $f_W$ satisfies
\[
\mathcal{Y}^*\left(f_W(w), x\right) = \mathcal{Y}(w, x) \;  \left(= \mathcal{Y}^*\left(\phi_{A_w}(w), x\right) \right) 
\]
for any $w \in W$.

\vspace{0.2\baselineskip}

The proof is complete as the converse is immediate.

\end{proof}

\subsection{ Properties and Examples of Internal Hom }

In this subsection, we first establish the left exactness of $\mathcal{H}(W, -)$ and $\mathcal{H}(-, W)$ for a restricted module $W$.
Then, we give some computational examples of the internal Hom.

Let $W, W^1, W^2$ be restricted $V$-modules. 
Let $f: W^1 \to W^2 $ be a $V$-homomorphism.
Then the linear map 
\[
 f_W: \H(W, W^1) \to \H(W, W^2)
\]
defined by
\[
f_W \big( \Y_A(a, t) \big)= f \Y_A(a, t),  \;  \; \; \forall \;  \Y_A(a, t) \in \H(W, W^1),
\]
is well-defined, and is a $V$-module homomorphism.

\begin{proposition}
Let $W, W^1, W^2, W^3$ be restricted $V$-modules, and suppose 
\[
0 \to W^1 \xrightarrow{f}  W^2   \xrightarrow{g}   W^3 
\]
is an exact sequence of $V$-modules.
Then the induced sequence of generalized $V$-modules
\[
0 \to \H(W, W^1) \xrightarrow{f_W} \H(W, W^2) \xrightarrow{g_W} \H(W, W^3)
\]
is exact.
\end{proposition}

\begin{proof}

It is trivial to verify that $f_W$  is injective and $\operatorname{Im}( f_W) \subset \operatorname{Ker} (g_W)$.
To prove that $ \operatorname{Ker} (g_W) \subset \operatorname{Im}( f_W) $, take $\Y_A(a, t) \in \operatorname{Ker} (g_W)$. 
Without loss of generality, we may assume that $A$ is generated by a single element $a$.
Since $g\Y_A(a,t)w=0$ for all $w\in W$ and $\operatorname{Im} f=\operatorname{Ker} g$,
we have 
\[
 \Y_A(a, t)w \in \operatorname{Im}(f)[\log x]\{x\}, \; \; \forall \; w \in W.
 \]
By Lemma \ref{gen-submodule-lemma}, it follows that
\[
  \Y_A( e , t)w \in \operatorname{Im}(f)[\log x]\{x\}, \; \; \forall \; e \in A,  \;  w \; \in W.
\]
Consequently, $\Y_A \in \mathcal{I}\binom{ \operatorname{Im}f }{A \ W}$.
Let $f^{-1}: \operatorname{Im} f \to W^1 $  denote the inverse isomorphism of $f$ on $\operatorname{Im} f$.
Then we have $f^{-1}\Y_A(a, t) \in \H(W, W^1),$ and
\[ 
 f_W( f^{-1}\Y_A(a, t) ) =\Y_A(a, t) \in  \operatorname{Im}( f_W).
\]
Hence, $ \operatorname{Ker} (g_W) \subset \operatorname{Im}( f_W) $. This completes the proof.

\end{proof}
\begin{remark}
    In Section \ref{sec5}, we are going to realize $\H(W^1, -)$ as a right adjoint of the tensor product functor $-\boxtimes W^1$ in some suitable category. In this category, it is automatically left exact by general categorical arguments.
\end{remark}

Similarly, let $W, W^1, W^2$ be restricted $V$-modules. 
Let $f: W^1 \to W^2 $ be a $V$-homomorphism.
Then the linear map 
\[
 f^W: \H(W^2, W) \to \H(W^1, W)
\]
defined by

\[
f^W \big( \Y_A(a, t) \big)=  \Y_A(a, t)f,  \;  \; \; \forall \;  \Y_A(a, t) \in \H(W^2, W),
\]
is well-defined, and is a $V$-module homomorphism.

\begin{proposition}
Let $W, W^1, W^2, W^3$ be restricted $V$-modules, and suppose 
\[
W^1 \xrightarrow{f}  W^2   \xrightarrow{g}  W^3 \to 0  
\]
is an exact sequence of $V$-modules. 
Then the induced sequence of generalized $V$-modules
\[
0 \to \H(W^3, W) \xrightarrow{g^W} \H(W^2, W) \xrightarrow{f^W} \H(W^1, W)
\]
is exact.
\end{proposition}

\begin{proof}

It is trivial to verify that $g^W$  is injective and $\operatorname{Im}( g^W) \subset \operatorname{Ker} (f^W)$.
To prove that $\operatorname{Ker} (f^W)  \subset   \operatorname{Im}( g^W)$, take $\Y_A(a, t) \in \operatorname{Ker} (f^W)$. 
Without loss of generality, we may assume that $A$ is generated by a single element $a$.
Then we have
\[
\Y_A(a, t)f(w^1)=0, \; \; \forall \; w^1 \in W^1.
\]
By Lemma \ref{gen-submodule-lemma}, this vanishing extends to all of $A$:
\[
\Y_A(e,t)f(w^1)=0, \; \; \forall \; e \in A,   \;  w^1 \in W^1.
\]
This condition allows us to define an intertwining operator $\bar{\Y}_A \in \mathcal{I} \binom{W}{A \ W^2/\operatorname{Im}f } $  
via the rule
\[
\bar{\Y}_A(e, t)(w^2 + \text{Im }f )= \Y_A(e, t)w^2, \; \; \forall \; e \in A, \;  w^2 \in W^2. 
\]

By exactness, $\operatorname{Im} f = \operatorname{Ker} g$,
so $g$ induces a canonical isomorphism $\bar{g}: W^2/ \operatorname{Im}f  \to W^3$; write \(\bar{g}^{-1}\) for its inverse.
Then
\[
\bar{\Y}_A(-, x)\bar{g}^{-1} \in \mathcal{I} \binom{W}{A \ W^3 }
\]
and
\[
g^W \big(  \bar{\Y}_A(a, t)\bar{g}^{-1}   \big)=\Y_A(a, t).
\]
Hence $\operatorname{Ker} (f^W)  \subset   \operatorname{Im}( g^W)$. The proof is complete.

\end{proof}

The following result was established in \cite{Li98} for the case without logarithmic terms. We give an alternative proof here.

  \begin{example}
    Let $ (W,Y_W) $ be a restricted $V$-module. Then  $ \H(V,W)\cong W $. 
    \end{example}
       
    \begin{proof}
    	We show that the pair $ (W,\Omega_{0}Y_W) $ satisfies the universal property of $ \mathcal{H}(V,W) $, 
        which implies $ \mathcal{H}(V,W)\cong W$. By Theorem \ref{Universal-H}, we only need to verify the universal property for restricted modules.
        For any restricted $V$-module $(H, Y_H)$, and $\Y \in  \mathcal{I}\binom{W}{H\ V}$,
    	we will show that there exists a unique $V$-module homomorphism $\varphi:H \to W$ such that
    	\begin{align} \label{id2}
    		\mathcal{Y}(h,x) = \Omega_{0}Y_W\big(\varphi(h),x\big)
    	\end{align}
    	for all $h\in H $.
    	
    	Note that $ \Omega_{-1}\mathcal{Y} $ is an intertwining operator of type $ \binom{W}{V\; H}$.
        Since $ L(-1)\mathbf{1}=0$, we have
    	\[
    	\frac{d}{dx}\Omega_{-1}\mathcal{Y}(\mathbf{1},x)
    	=\Omega_{-1}\mathcal{Y}\big(L(-1)\mathbf{1},x\big)=0.
    	\]
    	Thus, $\Omega_{-1}\mathcal{Y}(\mathbf{1},x)$ is independent of $x$, 
        defining a constant linear map $\mathbf{1}_{-1;0}:H\to W$. We set
    	\[
    	\varphi = \Omega_{-1}\mathcal{Y}(\mathbf{1},x) = \mathbf{1}_{-1;0}:H\to W.
    	\]
    	
    	We first verify that $ \varphi $ is a $V$-module homomorphism. 
        Using the Jacobi identity for $ \Omega_{-1}\mathcal{Y} $ and the fact that $ Y(v,x)\mathbf{1} \in V[[x]] $ for all $ v\in V $, we compute:
    	\begin{align*}
    		& Y_W(v,x_1)\varphi - \varphi Y_H(v,x_2) \\
    		&= Y_W(v,x_1)\Omega_{-1}\mathcal{Y}(\mathbf{1},x_2) - \Omega_{-1}\mathcal{Y}(\mathbf{1},x_2)Y_W(v,x_1) \\
    		&= \operatorname{Res}_{x_0} x_0^{-1}\delta\!\left(\frac{x_1-x_2}{x_0}\right) Y_W(v,x_1)\Omega_{-1}\mathcal{Y}(\mathbf{1},x_2) \\
    		&\quad - \operatorname{Res}_{x_0} x_0^{-1}\delta\!\left(\frac{-x_2+x_1}{x_0}\right) \Omega_{-1}\mathcal{Y}(\mathbf{1},x_2)Y_W(v,x_1) \\
    		&= \operatorname{Res}_{x_0} x_2^{-1}\delta\!\left(\frac{x_1-x_0}{x_2}\right) \Omega_{-1}\mathcal{Y}\big(Y(v,x_0)\mathbf{1},x_2\big) \\
    		&= 0.
    	\end{align*}
    	This proves that \(\varphi:H\to W\) is indeed a \(V\)-module homomorphism.
    	
    	Since \(\Omega_{-1}\mathcal{Y}(\mathbf{1},x)\) is independent of \(x\) and already a \(V\)-module homomorphism, we further obtain for all \(v\in V\) and \(h\in H\):
    	\begin{align*}
    		& Y_W(v,x_1)\Omega_{-1}\mathcal{Y}(\mathbf{1},x_2)h \\    	
    		&= \operatorname{Res}_{x_2} x_2^{-1}\delta\!\left(\frac{x_1-x_0}{x_2}\right) Y_W(v,x_1)\Omega_{-1}\mathcal{Y}(\mathbf{1},x_2)h \\
    		&= \operatorname{Res}_{x_2} \Bigg( x_0^{-1}\delta\!\left(\frac{x_1-x_2}{x_0}\right) - x_0^{-1}\delta\!\left(\frac{-x_2+x_1}{x_0}\right) \Bigg)
    		Y_W(v,x_1)\Omega_{-1}\mathcal{Y}(\mathbf{1},x_2)h \\
    		&= \operatorname{Res}_{x_2} x_0^{-1}\delta\!\left(\frac{x_1-x_2}{x_0}\right) Y_W(v,x_1)\Omega_{-1}\mathcal{Y}(\mathbf{1},x_2)h \\
    		&\quad - \operatorname{Res}_{x_2} x_0^{-1}\delta\!\left(\frac{-x_2+x_1}{x_0}\right) \Omega_{-1}\mathcal{Y}(\mathbf{1},x_2)Y_H(v,x_1)h \\
    		&= \operatorname{Res}_{x_2} x_2^{-1}\delta\!\left(\frac{x_1-x_0}{x_2}\right) \Omega_{-1}\mathcal{Y}\big(Y(v,x_0)\mathbf{1},x_2\big)h.
    	\end{align*}
    	In the above identity, the left-hand side is independent of \(x_0\), while the right-hand side involves only nonnegative integer powers of \(x_0\). Setting \(x_0=0\) yields
    	\begin{align*}
    		Y_W(v,x_1)\Omega_{-1}\mathcal{Y}(\mathbf{1},x_2)h
    		&= \operatorname{Res}_{x_2} x_2^{-1}\delta\!\left(\frac{x_1}{x_2}\right) \Omega_{-1}\mathcal{Y}(v,x_2)h \\
    		&= \Omega_{-1}\mathcal{Y}(v,x_1)h.
    	\end{align*}
    	We thus arrive at
    	\[
    	Y_W(v,x)\varphi(h) = \Omega_{-1}\mathcal{Y}(v,x)h
    	\]
    	for all \(v\in V\) and \(h\in H\).
        It follows that
        \[
        \Y = \Omega_0\big([id, \varphi, id] \circ Y_W \big)= [\varphi, id, id] \circ (\Omega_0 Y_W).
        \]
        Hence for all $h \in H$, we have
    	\[
    	\Omega_{0}Y_W\big(\varphi(h),x\big)v = \mathcal{Y}(h,x)v,
    	\]
    	which is exactly the required identity (\ref{id2}).
    	
    	It remains to show the uniqueness of $\varphi$. Suppose \(\sigma:H\to W\) is another \(V\)-module homomorphism satisfying  the identity (\ref{id2}). 
    	Define \(f=\varphi-\sigma\). Then
    	\[
    	\Omega_{0}Y_W\big(f(h),x\big)v=0
    	\]
    	for all \(v\in V\) and \(h\in H\). This gives
    	\[
    	e^{xL(-1)}Y_W(\mathbf{1},e^{\pi i}x)f(h)=0,
    	\]
    	which forces \(f(h)=0\) for all \(h\in H\). Hence \(\varphi=\sigma\), establishing uniqueness.
    	This completes the proof.
    \end{proof}  
    
\begin{example} \label{ex-dual}
	Let $W^1$ and $W^2$ be restricted $V$-modules. Then we have
	\[
		\H(W^1, W^2) \cong \H \bigl((W^2)^\prime, (W^1)^\prime \bigr).
	\]
\end{example}

\begin{proof}
    We shall prove that the pair $\left(\H(W^1, W^2), A_0 \Y^{\H} \right)$ satisfies the same universal property of
     $\H \big( ( W^2)^\prime, (W^1)^\prime  \big)$ stated in Theorem \ref{Universal-H}, which yields the desired isomorphism.
     
    Let $W$ be a restricted $V$-module, and let $\Y$ be an intertwining operator of type $\binom{  (W^1)^\prime }{ W \ (W^2)^\prime  }$.
    Then $A_{-1}\Y$ is an intertwining operator of type $\binom{  W^2 }{ W \ W^1  }$.
    By the universal property of $\H(W^1, W^2)$, there exists a unique $V$-module homomorphism
	$\varphi_W \colon W \to \mathcal{H}(W^1,W^2)$ such that
		\[
		\mathcal{Y}^{\mathcal{H}}\big(\varphi_W(w),x\big) = A_{-1}\mathcal{Y}(w,x)
		\]
		for all \(w\in W\).
  Applying $A_0$ to both sides yields 
    \[
   A_0 \mathcal{Y}^{\mathcal{H}}\big(\varphi_W(w),x\big) =\mathcal{Y}(w,x)
    \]
   for all \(w\in W\). Suppose another $V$-module homomorphism $\phi_W: W \to \mathcal{H}(W^1,W^2)$ satisfies the same identity:
  \[
   A_0 \mathcal{Y}^{\mathcal{H}}\big(\phi_W(w),x\big) =\mathcal{Y}(w,x)
  \]
    for all \(w\in W\).   
  Applying $A_{-1}$ to this identity yields 
  \[
  \mathcal{Y}^{\mathcal{H}}\big(\phi_W(w),x\big) = A_{-1}\mathcal{Y}(w,x)
  \]
 for all \(w\in W\). 
 So the uniqueness of $\varphi_W$ forces $\phi_W=\varphi_W$.
 The required universal property is verified. This completes the proof.
\end{proof}

\begin{remark}
Example \ref{ex-dual} can also be proved directly. Define a map 
\[
\varphi: \H(W^1, W^2) \to \H \bigl((W^2)^\prime, (W^1)^\prime \bigr)
\]
by
\[
\varphi (\Y_A(a, t))= A_0 \Y(a, t)   \; \; \; \text{for} \ \forall \; \Y_A(a, t) \in  \H(W^1, W^2).
\]
Since the operator $A_0$ is invertible, it is straightforward to verify that $\varphi$ is a well-defined 
and is an isomorphism of $V$-modules.

\end{remark}

The following example is given in \cite{Li98};  here we present a different proof. 
In Chapter \ref{sec5}, we will generalize this result.

\begin{example}\label{ex:C_2-rational-iso}

	Let $V$ be a $C_2$-cofinite, rational vertex operator algebra with nonnegative integer weights.
	Let $ S_1, S_2,\dots, S_m$ be a complete set of representatives of isomorphism classes of irreducible ordinary $V$-modules.
	Let $ W^1, W^2 $ be irreducible ordinary $V$-modules.
	Then $ \mathcal{H}(W^1, W^2)'$  is a direct sum of finitely many irreducible ordinary modules and is isomorphic to $ W^1 \boxtimes (W^2)'$.
    
\end{example}

\begin{proof}

Under our assumptions, it is a standard fact that the fusion rules for any three ordinary $V$-modules are finite (see, e.g., \cite{Li99b, ABD04}), and the tensor product of any two ordinary modules exists (see, e.g. \cite{YZ26}).

Since $\H(W^1, W^2)'$ is a semi-simple $V$-module, 
it is completely determined up to isomorphism by the multiplicities of the simple modules $S_i$ occurring in it.
By the universal property of $\big( \H(W^1,W^2), \Y^\H \big)$ (see Theorem \ref{interHom1} (2) ),
each simple module $S_i$ appears in $ \mathcal{H}(W^1,W^2)$ with multiplicity equal to the fusion rule $N(S^i, W^1; W^2)$.
Consequently, 
\[
\mathcal{H}(W^1,W^2) \cong \bigoplus_{i=1}^m S_i^{\oplus N(S_i, W^1; W^2)}
\]
is a direct sum of finitely many irreducible ordinary modules. 

Similarly, by the universal property of $W^1 \boxtimes (W^2)^\prime$, we have
\[
W^1 \boxtimes (W^2)^\prime \cong \bigoplus_{i=1}^m S_i^{\oplus N\big(W^1,(W^2)';S_i\big)}.
\]
By the properties of fusion rules, $N(W^1, (W^2)'; S_i) = N(S_i', W^1; W^2)$. Therefore,
	\begin{align*}
        \H(W^1, W^2)^\prime & \cong \left( \bigoplus_{i=1}^m S_i^{\oplus N\big(S_i,W^1;W^2\big)} \right)' \\
        & \cong  \left( \bigoplus_{i=1}^m (S_i')^{\oplus N\big(S_i',W^1;W^2\big)} \right)'  
        \cong  \bigoplus_{i=1}^m S_i^{\oplus N\big(S_i',W^1;W^2\big)}\\
        & \cong \bigoplus_{i=1}^m S_i^{\oplus N\big(W^1,(W^2)';S_i\big)} \cong W^1 \boxtimes (W^2)^\prime. \
	\end{align*}
This completes the proof.
\end{proof}

\subsection{ Relation to Li's $\Delta(W^1, W^2)$}

Li introduced $\Delta(W^1,W^2)$ via generalized intertwining operators in \cite{Li98}.
It is clear from the definition that $\Delta(W^1, W^2)$ can be naturally extended to the logarithmic setting.
For our purpose, we recall the (logarithmic) generalized intertwining operators from \cite{Li98, M08}.

%For this purpose, we recall the logarithmic version of generalized intertwining operators from \cite{M02}.

	\begin{definition} \label{G-inter}
		Let $V$ be a vertex operator algebra, and let $W^1$ and $W^2$ be weak $V$-modules.   
		A {\bf generalized logarithmic intertwining operator} from $W^1$ to $W^2$ is a linear map 
		\begin{align*}
			\Phi(t): \ \ &   W^1  \  \to  \ W^2 [\log t]  \{ t \} \\
			& w \ \to \  \Phi(t)w= \sum_{n \in \C} \sum_{ k=0}^{p_n} \Phi_{n;k} w t^{-1-n} (\log t)^k \in  W^2 [\log t]  \{t\},
		\end{align*}
		where $p_n\in \N$, satisfying the following conditions:
		
		\begin{enumerate}[{(G1)}]
			\item For each $w^1 \in W^1$ and $h \in \C,$ $\Phi_{h+n;k}w^1=0$  for all sufficiently large integers $n$, independently $k$. 
			
			\item $[L(-1), \Phi(t)]= \frac{d}{dt}\Phi(t)$;
			
			\item For any $v \in V$, there exists a nonnegative integer $k_{\Phi,v}$ such that
			\[
			(x -t)^{k_{\Phi,v}} Y_{W^2}(v, x)\Phi(t) = (x- t)^{k_{\Phi,v}} \Phi(t) Y_{W^1}(v, x). 
			\]
		\end{enumerate}	
        
Denote the space of all generalized logarithmic intertwining operators from $W^1$ to $W^2$ by $\G(W^1, W^2)$.
\end{definition}
	
Let $A$ be a weak $V$-module, and let $\Y_A$ be an intertwining operator of type  $\binom{W^2}{A \ W^1}$. 
Then $\Y_A(a, t) \in  \G(W^1, W^2) $ for all $a \in A$.
Assume now that $W^1, W^2$ are restricted $V$-modules. It follows that
\[
\H(W^1, W^2) \subset \G(W^1, W^2).
\]
	
As in \cite{Li98}, for $u \in V, n \in \Z$, and $\Phi(t) \in \G(W^1, W^2) $, we define
	\[
	u^{\G}_n \Phi(t) = \text{Res}_{x} \Big( (x-t)^n Y_{W^2}(u, x) \Phi(t)- (-t+x)^n\Phi(t)Y_{W^1}(u, x) \Big),
	\]
and set 
	\[
	Y_{\G}(u, x)\Phi(t)= \sum_{n \in \Z } u^{\G}_n \Phi(t) x^{-n-1}.
	\]
We further define the linear map
	\[
	\Y^{\G}(-,x): \G(W^1,W^2) \otimes W^1 \to W^2 [\log x] \{x\}  
	\]
by
	\[
	\Y^{\G} \big(\Phi(t), x \big)w^1=\Phi(x)w^1.
	\]

The theorem below, which appears in \cite{M08}, 
is the logarithmic analogue of Theorems 4.6 and 4.7 in \cite{Li98}.  
Its proof does not depend on whether logarithmic terms appear.
Recent work \cite{DLXY24} provides an alternative proof in the setting of twisted modules.
    
	\begin{theorem} \label{Th-G}
		Let $V$ be a vertex operator algebra, and let $W^1, W^2$ be weak $V$-modules. 
		Then the following statements hold:
		\begin{enumerate}[{(1)}]
    \item  With the above definitions, $\big(\G(W^1,W^2), Y_{\G} \big)$ forms a weak $V$-module,
       and $\Y^{\G}(-,x)$ is an intertwining operator of type 
			$\binom{W^2}{\G(W^1,W^2) \; W^1}$;
  
  \item The pair $ \big(\G(W^1,W^2), \Y^\G \big) $ satisfies the following universal property: 
        For any weak $V$-module $W$ and any intertwining operator $\Y_W$ of type $\binom{W^2}{W \ W^1}$, there exists a unique $V$-module homomorphism
			$\varphi_W \colon W \to \G(W^1, W^2)$ such that
			\[
			\mathcal{Y}(w, x)w^1 = \Y^{\G} \bigl(\varphi_W(w), x\bigr)w^1,
			\]
			for all $ w \in W,\ w^1 \in W^1  $.
		\end{enumerate}
	\end{theorem}

We now extend the definition of $\Delta(W^1, W^2)$ from \cite{Li98} to the logarithmic setting.
\begin{definition}
	Let $V$ be a vertex operator algebra, and let $W^1$ and $W^2$ be restricted $V$-modules.
	Define $\Delta(W^1,W^2)$ to be the sum of all restricted $V$-submodules inside $\G(W^1,W^2)$.
\end{definition}

\begin{theorem} \label{delta=h}
Let $V$ be a vertex operator algebra, and let $W^1, W^2$ be restricted $V$-modules.
Then
 \[
 \big( \H(W^1, W^2), \ \Y^\H \big) = \big(\Delta(W^1, W^2), \  \Y^\G \big). 
 \]

\end{theorem}

\begin{proof}

It is easy to see that the inclusion $\H(W^1, W^2) \subset \G(W^1, W^2)$, 
and the actions of $Y^\H$ and $Y^\G$ on $\H(W^1, W^2)$ coincide.
Since $ \mathcal{H}(W^1,W^2) $ is the sum of its restricted $V$-submodules, it follows that
$\H(W^1, W^2)$ is a submodule of $\Delta(W^1, W^2)$.

Now take an arbitrary element $\Phi(t) \in \Delta(W^1, W^2)$.
By the definition of $\Delta(W^1, W^2)$, there exists a restricted $V$-submodule $W \subseteq \Delta(W^1, W^2)$  containing $\Phi(t) $.
Consequently,
\[
    \Phi(t)=\Y^{\G}|_{W}\big(\Phi(x),t\big) \in \H(W^1,W^2).  
\]
This yields the equality of $V$-modules
\[
 \big( \H(W^1, W^2), \ \Y^\H \big) = \big(\Delta(W^1, W^2), \  \Y^\G \big). 
\]
The proof is complete.

\end{proof}

\section{The Finiteness of $\H(W^1, W^2)$  and Its applications} \label{sec5}

In this subsection, we show that $\mathcal{H}(W^1, W^2)$ is a restricted $V$-module assuming $C_1$-cofiniteness. 
Using this result, we show that $\mathcal{H}(W^1, W^2)^\prime$ and $W^1 \boxtimes (W^2)^\prime$ are isomorphic, which extends the result in Example \ref{ex:C_2-rational-iso}. 
In addition, $\mathcal{H}(W^1, W^2)$ can be interpreted as an internal Hom in some suitable tensor categories.

\subsection{The Finiteness of $\H(W^1, W^2)$ }

 \begin{theorem} \label{f-dim}
Let $W^1$ and $ (W^2)^\prime $ be $C_1$-cofinite restricted $V$-modules. 
Then $\H(W^1, W^2)$ is a restricted $V$-module. 
 Furthermore, its contragredient module $\mathcal{H}(W^1,W^2)'$ is also a $C_1$-cofinite restricted $V$-module.
\end{theorem}

\begin{proof}

By Theorem \ref{C1-tensor}, the tensor $W^1 \boxtimes (W^2)^\prime$ exists and is a $C_1$-cofinite restricted $V$-module.
Let $ W $ be an arbitrary restricted $ V$-submodule of $ \mathcal{H}(W^1,W^2)$. 
Then the restriction intertwining operator
\[
\Y^\H|_{W}(-,x)- \in \mathcal{I} \binom{W^2}{W\ W^1}
\]
is injective in its first argument. 
By Lemma \ref{Inj-Sur}(2), the intertwining operator  
\[
A_0 \Omega_0 \Y^\H|_{W} \in  \mathcal{I} \binom{ W^\prime }{W^1 \; (W^2)^{\prime}}
\]
is surjective. 
This implies that $W'$ is a quotient module of $W^1 \boxtimes (W^2)'$.
It follows that
\begin{align}
    \operatorname{wt}(W) = \operatorname{wt}(W') \subset \operatorname{wt}\big(W^1 \boxtimes (W^2)'\big),
\end{align}
and
\begin{align}
\dim W_{(\lambda)} = \dim W'_{(\lambda)} \leq \dim \big(W^1 \boxtimes (W^2)'\big)_{(\lambda)}
\end{align}
for all $\lambda\in\mathbb{C}$.

Since $ \mathcal{H}(W^1,W^2) $ is the sum of its restricted  $ V$-submodules, 
and the above two inequalies hold for every restricted $V$-submodule $W$ of $\H(W^1, W^2)$, we further obtain
\[
 \operatorname{wt}(\H(W^1, W^2)) \subset \operatorname{wt}\big(W^1 \boxtimes (W^2)'\big),
\]
and 
\[
\dim \H(W^1, W^2)_{(\lambda)}   \leq \dim \big(W^1 \boxtimes (W^2)'\big)_{(\lambda)}
\]
for all $\lambda \in \C$.
Consequently, \(\mathcal{H}(W^1,W^2)\) is a restricted \(V\)-module.

Now the operator $A_0 \Omega_0 \Y^\H $ is a well-defined surjective intertwining operator of type
\[
\mathcal{I} \binom{ \H(W^1, W^2)^\prime }{W^1 \; (W^2)^{\prime}}.
\]
Therefore, $\mathcal{H}(W^1,W^2)'$ is a quotient of $W^1 \boxtimes (W^2)'$, 
and hence $\mathcal{H}(W^1,W^2)'$ is a $C_1$-cofinite restricted $V$-module.

\end{proof}

\begin{corollary} \label{finite-fusion1}

Suppose  $W^1$ and $ (W^2)^\prime $ are $C_1$-cofinite restricted $V$-modules,  
and let $W$ be an arbitrary restricted $V$-module. 
Then the intertwining operator space $ \mathcal{I}\binom{W^2}{W \; W^1} $ is finite-dimensional.
\end{corollary}

\begin{proof}

By the universal property of $\H(W^1, W^2)$, we have a vector space isomorphism 
\[
\operatorname{Hom}_V\bigl(W,  \mathcal{H}(W^1,  W^2)\bigr) \cong \mathcal{I}\binom{W^2}{W\; W^1}.
\]
Thus,  $ \mathcal{I}\binom{W^2}{W\; W^1}$   is finite-dimensional if and only if $\operatorname{Hom}_V\bigl(W,  \mathcal{H}(W^1,  W^2)\bigr)$ is so. Since $W$ and $\H(W^1, W^2)$ are restricted $V$-modules, there is a canonical isomorphism 
\[
\operatorname{Hom}_V \big( W,  \H(W^1, W^2) \big)  \cong \operatorname{Hom}_V \big( \H(W^1, W^2)^\prime,  W^\prime \big).
\]

Since $\mathcal{H}(W^1, W^2)'$ is a $C_1$-cofinite restricted $V$-module (see Theorem \ref{f-dim}), 
it is finitely generated. Suppose that $\mathcal{H}(W^1, W^2)'$ is generated by
\[
X = \mathcal{H}(W^1, W^2)'_{(\lambda_1)} \oplus \cdots \oplus \mathcal{H}(W^1, W^2)'_{(\lambda_n)},
\]
for some $\lambda_i \in \C, 1\leq i\leq n$.
Let $ Y= W^{\prime}_{ (\lambda_1) } \oplus \cdots \oplus W^{\prime}_{ (\lambda_n) } $.
Then for any 
\[
f \in  \operatorname{Hom}_V\big( \H(W^1, W^2)^\prime, W^\prime \big),
\]
we have $f(X) \subset Y.$
Because  $ \H(W^1, W^2)^\prime $ is generated by $X$,
the restriction map
\[
\varphi: \operatorname{Hom}_V\bigl(\mathcal{H}(W^1,W^2)', W'\bigr) \longrightarrow \operatorname{Hom}_{\mathbb{C}}(X,Y),\qquad 
f \longmapsto f|_X,
\]
is injective.  
Moreover, since $\operatorname{Hom}_{\mathbb{C}}(X,Y)  $ is finite-dimensional, we deduce that
\[
\operatorname{Hom}_V \big( \H(W^1, W^2)^\prime,  W^\prime \big) < \infty.  
\]
Consequently, $\operatorname{Hom}_V \big( W, \H(W^1, W^2) \big)$ is finite-dimensional, so the intertwining operator space  $\mathcal{I}\binom{W^2}{W\;W^1}$ is also finite-dimensional. This completes the proof.

\end{proof}

\begin{corollary}

Suppose  $W^1$ and $ W^2 $ are $C_1$-cofinite restricted $V$-modules,  
and let $W$ be an arbitrary restricted $V$-module. 
Then the intertwining operator space $\mathcal{I}\binom{W}{W^1 \; W^2}$ is finite-dimensional.

\end{corollary}

\begin{proof}

 Since 
\[
\mathcal{I} \binom{W}{W^1 \; W^2} \cong \mathcal{I} \binom{ (W^2)^\prime }{W^\prime \; W^1},
\]
and $ (W^2)^ { \prime \prime} \cong W^2$ is $C_1$-cofinite,  the result follows immediately from Corollary \ref{finite-fusion1}.

\end{proof}

\begin{remark}

Let $W^1$, $W^2$ be $C_1$-cofinite restricted $V$-modules, and let $W^3$ be a restricted $V$-module.
The finiteness of the fusion rule $N(W^1,W^2;W^3)$ was established in \cite{H05a} under the assumption that $(W^3)'$ is $C_1$-cofinite,
and in \cite{YZ26} under the condition that $W^3$ is of finite length.

Furthermore, this result extends the finiteness results for fusion rules in the $C_2$-cofinite setting obtained in \cite{Li99b, ABD04}. 
Indeed, by Lemma \ref{lem:C2-C1}, if $V$ is a $C_2$-cofinite vertex operator algebra without negative conformal weights, then a weak $V$-module is ordinary if and only if it is $C_1$-cofinite.
The finiteness results for fusion rules in \cite{Li99b, ABD04} were derived via the $A(V)$-bimodule theory from \cite{FZ92, Li99a}.

\end{remark}

\begin{corollary}
Let $V$ be a $C_2$-cofinite vertex operator algebra without negative conformal weights.
Let $W^1$ and $W^2$ be restricted $V$-modules. Then
\[
\mathcal{H}(W^1, W^2) = \mathcal{G}(W^1, W^2),
\]
and this common module is a restricted $V$-module.
\end{corollary}

\begin{proof}
Since every finitely generated weak module is a restricted $V$-module (Lemma \ref{lem:C2-C1}), every weak module is the sum of its restricted submodules. Hence $\mathcal{H}(W^1, W^2)$ and $\mathcal{G}(W^1, W^2)$ coincide. Since every restricted $V$-module is $C_1$-cofinite, Theorem \ref{f-dim} implies that $\mathcal{H}(W^1, W^2)$ is a restricted $V$-module.
\end{proof}

\subsection{ Relations Between $\H(W^1, W^2)$ and Tensor Product }

\begin{theorem} \label{H-Tensor}
Let \(W^1\) and \((W^2)'\) be restricted \(V\)-modules.
Then $\mathcal{H} (W^1, W^2 )$ is a restricted \(V\)-module if and only if
the tensor product $W^1 \boxtimes (W^2)^\prime$ exists 
in the category of restricted \(V\)-modules.
Furthermore, under this equivalence, we have
\[
 \mathcal{H}(W^1,W^2)' \cong  W^1 \boxtimes (W^2)^\prime.
\]
\end{theorem}

\begin{proof}
($  \Rightarrow $) \;
Assume that $\mathcal{H} (W^1, W^2 )$ is a restricted \(V\)-module.
Then  $A_{0} \Omega_{0} \mathcal{Y}^{\mathcal{H}} $ 
is a well-defined intertwining operator of the type 
\[
\binom{( \H(W^1, W^2)^\prime}{W^1 \; (W^2)^{\prime}}.
\]

We prove that the pair 
\[
\big(\H(W^1, W^2)^\prime, \ A_{0} \Omega_{0} \mathcal{Y}^\H \big)
\]
satisfies the universal property defining the tensor product of $W^1$ and $(W^2)^{\prime}$.
It then follows that the tensor product $W^1 \boxtimes (W^2)^\prime$ exists and is isomorphic to $ \mathcal{H}(W^1,W^2)'$.

To verify this claim, let $W$ be an arbitrary restricted $V$-module, 
and let $\mathcal{I}$ be an intertwining operator
of type $\binom{W}{W^1 \ (W^2)^\prime}$. 
Then $ \Omega_{-1} A_{-1} \mathcal{I}$ is an intertwining operator of type $\binom{W^2}{W^\prime \ W^1}$. 
Consequently, there exists a $V$-homomorphism 
\[
\varphi: W^\prime \to \H(W^1, W^2) 
\]
such that 
\[
\Y^\H\big(\varphi (w^{\prime}), x  \big)= \Omega_{-1} A_{-1} \mathcal{I}(w^{\prime}, x )
\]
for all $w^\prime \in W^\prime$.
It follows that
\begin{align*}
 \mathcal{I} & =A_0 \Omega_0 \big([\varphi, id; id] \circ \Y^\H \big) =A_0  \Big([id, \varphi; id] \circ \big( \Omega_0\Y^\H \big) \Big)\\
  &= \Big([id, id; \varphi^*] \circ \big(A_0\Omega_0\Y^\H \big) \Big)= \varphi^* A_0\Omega_0\Y^\H.
\end{align*}
Moreover, since $A_0\Omega_0 \Y^\H$ is a surjective intertwining operator, 
the morphism $\varphi^*$ satisfying 
\[
\varphi^* A_0\Omega_0\Y^\H= \mathcal{I}
\]
is unique. 
Thus, the required universal property holds.

\vspace{0.5\baselineskip}

($ \Leftarrow $) \;  Assume that the tensor product $W^1 \boxtimes (W^2)^\prime$ exists 
in the category of restricted \(V\)-modules.
Then  $ \Omega_{0} A_{0}  \Y^{\boxtimes} $ 
is a well-defined intertwining operator of the type 
\[
\binom{W^2}{ \big(W^1 \boxtimes (W^2)^\prime \big)^\prime \; W^1}.
\]
We show that the pair 
\[
\big( \big(W^1 \boxtimes (W^2)^\prime \big)^\prime, \  \Omega_{0} A_{0}  \Y^{\boxtimes} \big)
\]
satisfies the universal property of $\H(W^1, W^2)$  given in Theorem \ref{Universal-H}.
Consequently, 
\[
\H(W^1, W^2) \cong \big(W^1 \boxtimes (W^2)^\prime \big)^\prime
\]
is a restricted $V$-module.

Let $W$ be any  restricted $V$-module, 
and let $\mathcal{I}$ be an intertwining operator
of type $\binom{W^2}{W \ W^1}$. 
Then $A_{-1} \Omega_{-1}  \mathcal{I}$ is an intertwining operator of type $\binom{W^\prime}{W^1 \ (W^2)^\prime}$. 
Consequently, there exists a $V$-homomorphism 
\[
\varphi: W^1 \boxtimes (W^2)^\prime  \to W^\prime  
\]
such that 
\[
\varphi  \Y^\boxtimes = A_{-1} \Omega_{-1}  \mathcal{I}.
\]
It follows that
\begin{align*}
 \mathcal{I} & = \Omega_0 A_0 \big([id, id; \varphi] \circ \Y^\boxtimes \big) =
 \Omega_0 \Big([id, \varphi^*;  id ] \circ \big( A_0 \Y^\boxtimes \big) \Big)\\
  &= [\varphi^*, id; id] \circ \big(\Omega_0 A_0 \Y^\boxtimes \big) .
\end{align*}
Thus, for all $w \in W$,
\[
\mathcal{I}(w, x)=  \Omega_0 A_0 \Y^\boxtimes \big( \varphi^*(w), x \big).
\]
Moreover, since $  \Omega_0 A_0 \Y^\boxtimes $ is injective in its first argument, 
the morphism $\varphi^*$  satisfying the above identity is unique. This verifies the universal property of  $\H(W^1, W^2)$.

\end{proof}

\begin{corollary} \label{C1H=Tensor}
Let $W^1$ and $ (W^2)^\prime $ be $C_1$-cofinite restricted $V$-modules.
We have a restricted $V$-module isomorphism 
\[
\H (W^1, W^2 )^\prime \cong W^1 \boxtimes (W^2)^\prime.
\]
\end{corollary}

\begin{proof}
It follows from Theorem \ref{f-dim} and Theorem \ref{H-Tensor}.
\end{proof}

\subsection{ Realization of the Internal Hom in Monoidal Categories } \label{internal-hom}

Let $(\mathcal{C}, \otimes)$ be a monoidal category, and $W^1, W^2 \in \mathcal{C}$.
If there exists an object $\underline{\operatorname{Hom}}(W^1, W^2)$ in $  \mathcal{C}$, and a natural isomorphism
\[
\operatorname{Hom}_{\mathcal{C}} \big(X, \; \underline{\operatorname{Hom}}(W^1, W^2) \big) \cong \operatorname{Hom}_{\mathcal{C}}( X \otimes W^1, \; W^2), 
\]
for all $X \in \mathcal{C}$, then $\underline{\operatorname{Hom}}(W^1, W^2)$  is called the {\bf internal Hom } from $W^1$ to $W^2$ \cite{EGNO15}. 

Let $V$ be a vertex operator algebra such that a restricted $V$-module is $C_1\mbox{-cofinite}$ if and only if it has finite composition length. 
By Lemma \ref{lem:C2-C1}, this condition holds whenever $V$ is a $C_2\mbox{-cofinite}$ vertex operator algebra with no negative conformal weights.
Denote by $\mathcal{C}_V$ the category of restricted $V$-modules of finite composition length.
Then $\mathcal{C}_V$ is a $\C$-linear abelian category, and it is closed under tensor products and contragredients.
We assume that $( \mathcal{C}_V, \boxtimes )$  constitutes a monoidal category.
In fact, under the hypotheses on $V$,  $( \mathcal{C}_V, \boxtimes )$  is a tensor category, as established in \cite{McR24,CMHFY26}. 
In particular, if $V$ is a regular vertex operator algebra, then $\mathcal{C}_V$ is even a modular tensor category  under certain
additional conditions \cite{H08}.

\begin{theorem} \label{inter-hom-th}
Let $W^1, W^2 \in \mathcal{C}_V$. Then  $\mathcal{H}(W^1,W^2)$ is the internal Hom from $W^1$ to $W^2$ in $\mathcal{C}_V$.
\end{theorem}

\begin{proof}	
By Theorem \ref{f-dim}, $\mathcal{H}(W^1,W^2)^\prime$ is $C_1$-cofinite. 
Thanks to our hypotheses on $V$, $\mathcal{H}(W^1,W^2)^\prime$ has finite composition length.
It follows that $\mathcal{H}(W^1,W^2)$ likewise has finite composition length, hence $\mathcal{H}(W^1, W^2) \in \mathcal{C}_V$.
From the universal properties satisfied by $W^1 \boxtimes W^2$ and $\H(W^1, W^2)$, we obtain the following isomorphism:
\begin{align*}
    & \theta_{X; W^1,W^2}: \Hom_V \left(X, \; \H(W^1, W^2) \right) \to \Hom_V( X \boxtimes W^1, \; W^2) \\
    & \hspace{8em} \varphi \hspace{5em} \longmapsto \hspace{2em} \theta_{X; W^1,W^2}(\varphi),
\end{align*}
where $\theta_{X; W^1,W^2}(\varphi)$ is uniquely determined by 
\[ 
\theta_{X; W^1,W^2}(\varphi)  \mathcal{Y}^{X\boxtimes W^1}\big(-,x\big) = \mathcal{Y}^{\H(W^1, W^2)}\big( \varphi(-), x \big).
\] 
Here, for clarity, we use $\mathcal{Y}^{\H(W^1, W^2)}$ and $\mathcal{Y}^{X\boxtimes W^1}$ 
to denote the canonical intertwining operator in 
$\mathcal{I}\binom{W^2}{\H(W^1, W^2) \ W^1}$ and $\mathcal{I}\binom{W\boxtimes W^1}{X \ W^1}$, respectively. 
To show that $\theta_{X; W^1,W^2}$ is natural in $X$, let $Y \in \mathcal{C}_V$ and $f\in \Hom_V(Y, X)$. 
We need to show that the following diagram is commutative:
\begin{equation*}
    \begin{tikzcd}
        \Hom_V\left(X, \; \H(W^1, W^2) \right)\arrow[r, "\theta_{X; W^1,W^2}"]\arrow[d, "\circ f"] & \Hom_V( X \boxtimes W^1, \; W^2)\arrow[d, "\circ (f\boxtimes id_{W^1})"]\\
        \Hom_V\left(Y, \; \H(W^1, W^2) \right)\arrow[r, "\theta_{Y; W^1,W^2}"] & \Hom_V( Y \boxtimes W^1, \; W^2),
    \end{tikzcd}
\end{equation*}
where $f\boxtimes id_{W^1}\in \Hom_V(Y\boxtimes W^1, X\boxtimes W^1)$ is the unique homomorphism satisfying 
\[
\big( f\boxtimes id_{W^1} \big) \mathcal{Y}^{Y\boxtimes W^1}\big(-, x \big) = \mathcal{Y}^{X\boxtimes W^1}\big(f(-), x \big).
\]  
Now for any $\varphi \in \Hom_V\left(X, \; \H(W^1, W^2) \right)$, we have
\[
\theta_{Y; W^1,W^2}(\varphi \circ  f) \mathcal{Y}^{Y\boxtimes W^1}(-, x) = \mathcal{Y}^{\H(W^1, W^2)}\big(  (\varphi \circ f)(-), x \big),
\]
and
\begin{align*}
\big( \theta_{X; W^1,W^2} (\varphi)  \circ (f\boxtimes id_{W^1}) \big)  \mathcal{Y}^{Y\boxtimes W^1}(-, x) 
&=\theta_{X; W^1,W^2}(\varphi) \mathcal{Y}^{X\boxtimes W^1} \big( f(-), x \big) \\
&=  \mathcal{Y}^{\H(W^1, W^2)}\big( ( \varphi \circ f) (-), x \big).
\end{align*}
Thus, 
\[
\theta_{Y;W^1,W^2}(\varphi \circ f) = \theta_{X; W^1,W^2}(\varphi) \circ (f\boxtimes id_{W^1}),
\]
i.e., the diagram is commutative. 
Consequently, $\H(W^1, W^2)$ is the internal Hom from $W^1$ to $W^2$ in $\mathcal{C}_V$. This completes the proof.

\end{proof}

\begin{remark}
Similarly, one can show that the isomorphism $\theta_{X; W^1, W^2}$ is natural in $W^2$.
Hence, the two functors $-\boxtimes W^1$ and $\H(W^1, -)$ form an adjoint pair.
\end{remark}

\begin{remark}
Let $W^1, W^2 \in \mathcal{C}_V$. By the  universal  property of the tensor product, there exists a unique $V$-module homomorphism 
\[ 
\text{ev}_{W^1, W^2}: \H(W^1, W^2) \boxtimes W^1 \to W^2 
\]
such that $ \text{ev}_{W^1, W^2} \circ \Y^\boxtimes=\Y^\H.$
In fact,  $\mathrm{ev}_{W^1, W^2}$ is precisely the image of $\mathrm{id}_{\H(W^1, W^2)}$ under the natural isomorphism above.
Let $W^1, W^2, W^2 \in \mathcal{C}_V$, then we have the following canonical composition:
\begin{align*}
\big( \H(W^2, W^3)  \boxtimes \H(W^1, W^2) \big) \boxtimes W^1 & \cong \H(W^2, W^3) \boxtimes \big( \H(W^1, W^2) \boxtimes W^1  \big) \\
 & \to \H(W^2, W^3) \boxtimes W^2 \to W^3.
\end{align*}
By the universal property of $\mathcal{H}(W^1, W^3)$, this composition produces a (unique) $V$-module homomorphism 
\[
m: \H(W^2, W^3)  \boxtimes \H(W^1, W^2) \to \H(W^1, W^3).
\]
By the  universal  property of the tensor product,
The homomorphism $m$ corresponds uniquely to an intertwining operator $\mathcal{F}$ of type 
\[
\mathcal{I} \binom{\H(W^1, W^3)  }{\H(W^2, W^3) \;  \H(W^1, W^2) }.
\]
We note that the existence of the $V$-homomorphism $m$ 
is uniquely determined by the properties of the internal Hom in tensor categories \cite{EGNO15}. 
A natural question then arises: how can we construct the intertwining operator $\mathcal{F}$ canonically from the perspective of vertex operator algebras?

\end{remark}
   
\section{  Isomorphism Between Internal Hom and  $P(z_0)$-dual product  }  \label{sec6}

Fix a simply connected domain $U \subset \C-\{0\}$, let $\operatorname{Log} z: U \to \C $ denote a fixed holomorphic branch of the complex logarithm on $U$.
For any complex number $n \in \C$, we define the complex power function $z^n: U \to \C$ via
\[
z^n = e^{n\operatorname{Log} z}, \ \ \ z \in U.
\]
Under this definition, $\operatorname{Log} z$ and $z^{\alpha}$ are holomorphic on $U$.

Let $W^1, W^2, W^3$ be generalized $V$-modules, and let $w^1 \in W^1, w^2 \in W^2$ and $ w_3^\prime \in  (W^3)^\prime$.
For any intertwining operator $\Y$ of type  $\binom{W^3}{W^1   W^2}$, the series
\[
 \langle w_3^\prime, \  \Y(w^1, x)w^2 \rangle=    \sum_{n, k} \langle w_3^\prime,\  w^1_{n;k}w^2 \rangle x^{-n-1} (\log x)^k 
\] is in fact a finite sum \cite[Remark 3.16]{YZ26}. Then, we can define its evaluation at $z \in U$ by
\begin{align*}
 \langle w_3^\prime, \;  \Y(w^1, z)w^2 \rangle
\overset{\mathrm{def}} { =}    \sum_{n, k} \langle w_3^\prime, \;  w^1_{n;k}w^2 \rangle e^{(-n-1)\operatorname{Log} z} (\operatorname{Log} z)^k. 
\end{align*}
Then for any $w^1 \in W^1, w^2 \in W^2$ and $ w_3^\prime \in  (W^3)^\prime$, 
the function 
\[
\langle w_3^\prime, \;  \Y(w^1, z)w^2 \rangle: U \to \C
\]
is holomorphic on $U$.

\vspace{0.5\baselineskip}

The following analytic lemma is standard; for a proof, see \cite{YZ26}.

\begin{lemma} \label{0=0}

	Let $\lambda_1, \lambda_2, \cdots, \lambda_n \in \mathbb{C}$ be pairwise distinct complex numbers, 
    and $k_1, k_2, \cdots, k_m \in \mathbb{N}$ pairwise distinct nonnegative integers.
	 Define the holomorphic function on $U$ by
	\[
	F(z)=\sum_{i=1}^{n}\sum_{j=1}^{m}a_{i,j} e^{ \lambda_i \operatorname{Log} z } (  \operatorname{Log} z )^{k_j}, \; \; a_{i, j} \in \C.
	\] 
    If $F(z) \equiv 0 $ on some non-empty open subset $O \subset U$, then $a_{i,j}=0$ for all $i, j$.
\end{lemma}

The lemma below was employed in \cite{H25C1} to prove the $C_1$-cofiniteness of tensor products of $C_1$-cofinite modules; 
here we reorganize and refine it.

\begin{lemma} \label{keylemma}
Let $W^1, W^2, W^3$ be generalized $V$-modules, $w_3^\prime \in (W^3)^\prime $, 
and let $ \mathcal{Y} $ be an intertwining operator of type  $ \binom{W^3}{W^1\,W^2}$. 
Fix $z_0 \in U$.
If
\[
 \bigl\langle w_3^\prime, \;   \mathcal{Y}(w^1, z_0)w^2 \bigr\rangle = 0
\]
for all $w^1 \in W^1$ and $ w^2 \in W^2$,
then
\[
 \bigl\langle w_3^\prime, \;   \mathcal{Y}(w^1, x)w^2 \bigr\rangle = 0
\]
for all $w^1 \in W^1$ and $ w^2 \in W^2$. 

\end{lemma}

\begin{proof}

Take any $w^1 \in W^1, w^2 \in W^2$, and define
\[
f\big(w^1, w^2; z\big)= \bigl\langle w_3^\prime, \;   \mathcal{Y}(w^1, z)w^2 \bigr\rangle.
\]
Then $f\big(w^1,w^2;z\big)$ is holomorphic on $U$.
By the $L(-1)$-derivative property of $\Y$, for any $k \in \N$ we have
\begin{align*}
(\frac{d}{dz})^k f\big(w^1, w^2; z\big) {\big |}_{z =z_0}
 &=(\frac{d}{dx})^k \bigl\langle w_3^\prime, \;   \mathcal{Y}(w^1, x)w^2 \bigr\rangle {\big |}_{ x^n(\log x)^k=e^{ n \operatorname{Log} z_0}(\operatorname{Log} z_0)^k} \\
 &=\bigl\langle w_3^\prime, \;   \mathcal{Y} \big( L(-1)^kw^1, x \big)w^2 \bigr\rangle {\big |}_{ x^n(\log x)^k=e^{ n \operatorname{Log} z_0}(\operatorname{Log} z_0)^k} \\
 &= \bigl\langle w_3^\prime, \;   \mathcal{Y} \big( L(-1)^kw^1, z_0 \big)w^2 \bigr\rangle=0.
\end{align*}
Thus, there exists an open domain $O_{w_1, w_2} \subset U$ containing $z_0$ such that $f\big(w^1, w^2; z \big)=0$ 
for all $z \in O_{w_1, w_2}$.
By Lemma \ref{0=0}, we obtain $\langle w_3^\prime, \;  w^1_{n;k}w^2 \rangle =0$  for all $n \in \C$  and $ k \in \N$.
It follows that
\[
 \bigl\langle w_3^\prime, \;   \mathcal{Y}(w^1, x)w^2 \bigr\rangle = 0.
\]
The proof is complete.

\end{proof}

We review the recent constructions of  $W^1 \pzbox_{P(z_0)} W^2$ in the untwisted setting  given in \cite{DH25} and \cite{H26}. 
Note that \cite{DH25} defines intertwining operators via analytic methods, whereas \cite{H26} uses the twisted Jacobi identity. 
In the special case $g_1=g_2=\text{id}_V$ of \cite{H26}, the resulting space $W^1 \pzbox_{P(z_0)} W^2$ coincides with our definition given below.

Fix $z_0 \in U$. Let $V$ be a vertex operator algebra, and let $W^1, W^2$ be restricted $V$-modules.
Given a restricted $V$-module $W$, an intertwining operator $\mathcal{Y}_W$ of type $\binom{W}{W^1\,W^2}$,
and $w'\in W'$, we define an element $\boldsymbol{\lambda}_{\mathcal{Y}_W,w'}\in (W^1\otimes W^2)^*$ by
\[
\boldsymbol{\lambda}_{\mathcal{Y}_W,w'}(w^1\otimes w^2)
= \langle w', \; \mathcal{Y}_W(w^1,z_0) w^2 \rangle \in \mathbb{C},
\]
for all $w^1\in W^1$ and $w^2\in W^2 $.

We then define $W^1\pzbox_{P(z_0)} W^2 $ to be the subspace of $(W^1\otimes W^2)^*$
spanned by all such elements $\boldsymbol{\lambda}_{\mathcal{Y}_W,w'}$,
where $W$ runs over all restricted $V$-submodules,
$\mathcal{Y}_W$ runs over all intertwining operators of type $\binom{W}{W^1\,W^2}$,
and $w'$  ranges over $ W'$. We call $W^1 \pzbox_{P(z_0)} W^2$ the {\bf $P(z_0)$-dual product } of $W^1$ and $W^2$.
By a similar argument as in the definition of $\H(W^1, W^2)$, every element in $W^1 \pzbox_{P(z_0)} W^2$ can be expressed as 
$\boldsymbol{\lambda}_{\mathcal{Y}_W,w'} $ for some restricted $V$-module $W$, $ w^\prime \in W^\prime$, 
and intertwining operator $ \mathcal{Y}_W \in \mathcal{I}\binom{W}{W^1 \;W^2}$. The following result was obtained in \cite{DH25}. We include a proof of the well-definedness of the module action for completeness.

\begin{lemma} 
Let $u \in V$, $n \in \mathbb{Z}$, and $ \boldsymbol{\lambda}_{\mathcal{Y}_W,w'} \in W^1\pzbox_{P(z_0)} W^2$. Define an action
\[
u_n \boldsymbol{\lambda}_{\mathcal{Y}_W,w'} = \boldsymbol{\lambda}_{\mathcal{Y}_W, u_nw'}  \in W^1\pzbox_{P(z_0)} W^2.
\]
Then this action is well-defined, and $W^1\pzbox_{P(z_0)} W^2$ is a generalized $V$-module under this action.
\end{lemma}

\begin{proof}

We first prove that this action is well-defined. 
Suppose that $\boldsymbol{\lambda}_{\mathcal{Y}_W,w'} =\boldsymbol{\lambda}_{\mathcal{Y}_M,m'}$ as elements of $W^1 \pzbox_{P(z_0)} W^2$.
By definition, this implies
\[
\langle w', \; \mathcal{Y}_W(w^1,z_0) w^2 \rangle = \langle m', \; \mathcal{Y}_M(w^1,z_0) w^2 \rangle,  
\; \; \; \; \forall \; w^1 \in W^1, \; w^2 \in W^2.
\]
Consider the direct sum intertwining operator 
\[
\Y_W \oplus \Y_M \in \mathcal{I} \binom{W \oplus M }{W^1 \; W^2 }. 
\]
We canonically identify the pair $(w^\prime, -m^\prime)$ with an element of $( W \oplus M)^\prime$.
It follows that 
\[
 \big  \langle (w^\prime, -m^\prime), \; \big(\Y_W \oplus \Y_M \big) (w^1,z_0) w^2   \big  \rangle =0  \; \; \; \; \forall \; w^1 \in W^1, \; w^2 \in W^2.
\]
Applying Lemma \ref{keylemma}, we deduce
\[
\langle w', \; \mathcal{Y}_W(w^1, x) w^2 \rangle = \langle m', \; \mathcal{Y}_M(w^1, x) w^2 \rangle \; \; \; \;
\forall \; w^1 \in W^1, \; w^2 \in W^2.
\]
By the Jacobi identity for $\Y_W$, the following identity holds for all $w^1 \in W^1, w^2 \in W^2$, and $u \in V$:
\begin{align*}
& \left \langle w', \;  Y_W(u, x_1) \mathcal{Y}_W(w^1, x) w^2 \right \rangle \\
&= \left \langle w', \;  \mathcal{Y}_W(w^1, x) Y_{W^2}(u, x_1)  w^2 \right \rangle 
 +\Res_{x_0}  x^{-1} \delta\!\left(\frac{x_1-x_0}{x}\right)  
 \left \langle w', \;  \Y_W \left ( Y_{W^1}(u, x_0)  w^1, x \right)   w^2 \right \rangle \\
&= \left \langle m', \;  \mathcal{Y}_M(w^1, x) Y_{W^2}(u, x_1)  w^2 \right \rangle 
 +\Res_{x_0}  x^{-1} \delta\!\left(\frac{x_1-x_0}{x}\right) 
 \left \langle m', \;  \Y_M \left( Y_{W^1}(u, x_0)  w^1, x \right)   w^2 \right \rangle\\
&=  \left \langle m', \;  Y_M(u, x_1) \mathcal{Y}_M(w^1, x) w^2 \right \rangle
\end{align*}
It then follows that
\begin{align*}
\langle  u_nw', \; \mathcal{Y}_W(w^1, x) w^2 \rangle  
& =\Res_{x_1} x_1^n \left \langle Y_{W^\prime}(u, x_1) w', \; \mathcal{Y}_W(w^1, x) w^2 \right \rangle \\
&  =\Res_{x_1} x_1^n \left \langle  w', \; Y_W\big( e^{x_1L(1)} (-x_1^{-2})^{L(0)}u,  x_1^{-1}  \big) \mathcal{Y}_W(w^1, x) w^2 \right \rangle\\
&  =\Res_{x_1} x_1^n \left \langle  m', \; Y_M\big( e^{x_1L(1)} (-x_1^{-2})^{L(0)}u,  x_1^{-1}  \big) \mathcal{Y}_M(w^1, x) w^2 \right \rangle \\
&  = \Res_{x_1} x_1^n \left \langle  Y_{M^\prime}(u, x_1)  m', \; \mathcal{Y}_M(w^1, x) w^2 \right \rangle \\
&  = \langle u_nm', \; \mathcal{Y}_M(w^1, x) w^2 \rangle.
\end{align*}
valid for all $w^1 \in W^1, w^2 \in W^2$.
Consequently, we have
\[
\langle u_nw', \; \mathcal{Y}_W(w^1, z_0) w^2 \rangle = \langle u_nm', \;  \mathcal{Y}_M(w^1, z_0) w^2 \rangle \; \; \; \;
\forall \; w^1 \in W^1, \; w^2 \in W^2.
\]
This precisely gives $ \boldsymbol{\lambda}_{\mathcal{Y}_W, u_nw'}= \boldsymbol{\lambda}_{\mathcal{Y}_M, u_nm'}$.
Therefore, the action of $u_n$ on $W^1 \pzbox_{P(z_0)} W^2$ is well-defined. 
Now since each $W^\prime$ is a restricted $V$-module, 
$W^1\pzbox_{P(z_0)} W^2 $ carries a natural generalized $V$-module structure with respect to this action.

This completes the proof.

\end{proof}

The $P(z_0)$-dual product provides a natural realization of tensor products. 
According to \cite[Theorem 4.6]{DH25}, if $W^1 \pzbox_{P(z_0)} W^2$ is a restricted $V$-module, 
then the tensor product $W^1 \boxtimes W^2$ exists in the category of restricted $V$-modules, and satisfying
\[
W^1 \boxtimes W^2 \cong \big( W^1 \pzbox_{P(z_0)} W^2 \big)^\prime.
\]
Moreover, assuming that $W^1, W^2$ are $C_1$-cofinite restricted $V$-modules. 
It was shown in  \cite[Theorem 5.6]{H26} that $  W^1 \pzbox_{P(z)}W^2 $ is again a restricted $V$-module.
Combining these results with Corollary \ref{C1H=Tensor}, we obtain the following $V$-module isomorphism:
\[
\mathcal{H}\big(W^1, (W^2)^\prime \big) \cong (W^1 \boxtimes W^2)^\prime \cong W^1 \pzbox_{P(z)}W^2.
\]
This immediately gives
\[
\H \big(W^1, (W^2)^\prime \big) \cong W^1 \pzbox_{P(z)}W^2.
\]
The following theorem shows that this isomorphism remains valid even without $C_1$-cofiniteness:

\begin{theorem}
	Let  $V$ be a vertex operator algebra, and let $W^1, W^2$ be restricted $V$-modules.
	Then $\H \big(W^1, (W^2)^\prime \big)  $   and  $   W^1\pzbox_{P(z_0)} W^2$ are isomorphic as generalized $V$-modules.
\end{theorem}
	
	\begin{proof}

  Define a linear map 
    \[
    \Phi:   \H \big(W^1, (W^2)^\prime \big)  \to  W^1 \pzbox_{P(z_0)} W^2
    \]
   via the assignment
    \[
    \Phi \big( \Y_A(a, t) \big) = \boldsymbol{\lambda}_{  A_0 \Omega_0\Y_A, \ a },
    \]
   for every $\Y_A(a, t)  \in  \H \big(W^1, (W^2)^\prime \big)$. 
   Note that the subscript $a$ appearing in $\boldsymbol{\lambda}_{  A_0 \Omega_0\Y_A, \; a } $ is canonically regarded as an element of $A^{\prime \prime}$.
    
	First, we verify that $\Phi$ is well-defined. 
    Assume $\Y_A(a, t) $ and $\Y_B(b, t) $ define the same element in $\H \big(W^1, ( W^2)^\prime \big)$.
    By the definitions of $A_0 \Y$ and $\Omega_0 \Y$, for all $w^1 \in W^1$ and $w^2 \in W^2$ we have
    \begin{align*}
     & \Phi \Big( \Y_A(a, t) \Big) (w^1 \otimes w^2) = \Big\langle  A_0 \Omega_0\Y_A(w^1, z_0)w^2, \ a  \Big \rangle \\
     &= \Big\langle  A_0 \Omega_0\Y_A(w^1, x)w^2, \ a  \Big \rangle {\big |}_{  x^n (\log x)^k=e^{ n \operatorname{Log} z_0}(\operatorname{Log} z_0)^k }  \\
     &=\Big\langle w^2, \ e^{x^{-1} L(-1)} \Y_A 
     \big(a, e^{\pi \mathrm{i}}x^{-1} \big) e^{xL(1)} \big( e^{\pi \mathrm{i}} x^{-2} \big)^{L(0)}w^1    
     \Big \rangle { \big | }_{  x^n (\log x)^k=e^{ n \operatorname{Log} z_0}(\operatorname{Log} z_0)^k }\\
     &=\Big\langle w^2, \ e^{x^{-1} L(-1)} \Y_B 
     \big(b, e^{\pi \mathrm{i}}x^{-1} \big) e^{xL(1)} \big( e^{\pi \mathrm{i}} x^{-2} \big)^{L(0)}w^1    
     \Big \rangle {\big |}_{ x^n(\log x)^k=e^{ n \operatorname{Log} z_0} (\operatorname{Log} z_0)^k}\\
     &= \Big\langle  A_0 \Omega_0\Y_B(w^1, x)w^2, \ b  \Big \rangle {\big |}_{ x^n(\log x)^k=e^{ n \operatorname{Log} z_0}(\operatorname{Log} z_0)^k }  \\
     &=\Big\langle  A_0 \Omega_0\Y_B(w^1, z_0)w^2, \ b  \Big \rangle
     = \Phi \Big( \Y_B(b, t) \Big) (w^1 \otimes w^2).
    \end{align*}
     Consequently, 
     \[
     \Phi \Big(  \Y_A(a, t) \Big)=\Phi \Big( \Y_B(b, t) \Big).
     \]  
    This proves that $\Phi$ is well-defined.  
   From the definitions of the actions 
   $u_n$ on  $\H \big(W^1, (W^2)^\prime \big)$ and on $ W^1\pzbox_{P(z)} W^2$,  we know that $\Phi $ is a $V$-homomorphism.

    Next, we show that $\Phi$ is surjective. 
    Take an arbitrary element $\boldsymbol{\lambda}_{\mathcal{Y}_W, w'} \in  W^1\pzbox_{P(z)} W^2, $
	where $W$ is a restricted \(V\)-module, $w^\prime \in W^\prime $, and $\Y_W $ is an intertwining operator of type $ \binom{W}{W^1\,W^2} $. 
    One has
    \[
    \big( \Omega_{-1}A_{-1}\Y_W \big) (w^\prime, t) \in \H\big(W^1, (W^2)^\prime \big), 
    \]
	and the identity
    \[
    \Phi \Big(  \big( \Omega_{-1}A_{-1}\Y_W \big) (w^\prime, t) \Big) = \boldsymbol{\lambda}_{\mathcal{Y}_W, w'}.
    \]
	holds. Hence $\Phi$ is surjective.

   It remains to establish injectivity of $\Phi$. Assume that $\Phi \Big( \Y_A(a, t) \Big) =0.$
   Then
    \[
     \Big\langle A_0 \Omega_0\Y_A \big(w^1, z_0 \big)w^2, \; a  \Big \rangle  =0,
    \]
 for all $w^1 \in W^1$ and $w^2 \in W^2$. 
By Lemma \ref{keylemma}, we have 
\[
 \Big\langle A_0 \Omega_0\Y_A \big(w^1, x \big)w^2, \; a  \Big \rangle  =0
\]
for all $w^1 \in W^1$ and $w^2 \in W^2$.
It follows from Lemma \ref{Inj-Sur}(1) that
\[
 \Big\langle \Y_A \big(a, x \big)w^1, \; w^2  \Big \rangle  =0
\]
for every $w^1 \in W^1$ and $w^2 \in W^2$.
This forces $\Y_A(a, t)=0$. So, $\Phi$ is injective.

Therefore, $\Phi$ is a $V$-module isomorphism, which finishes the proof.        		
\end{proof}

\section*{Acknowledgement}

The first author is supported by NSFC (No. 12301039) and the Sichuan Science and Technology Program (No. 2025ZNSFSC0797).

%\section*{Declarations}

%The authors have no competing interests to declare that are relevant to the content of this article.
%No data was used for the research described in the article.

% \bibliographystyle{alpha}
% \bibliography{voa}

\end{document}